\documentclass{siamltex}
\usepackage{amsfonts}
\usepackage{amssymb}       
\usepackage{amsmath}
\usepackage{latexsym}
\usepackage{verbatim}      %
\usepackage{bm,enumerate}
\usepackage{multirow}
\usepackage{array}
\usepackage[dvips]{graphicx}
\usepackage[hang]{subfigure}
\usepackage{float}
\usepackage{rotating}
\usepackage{color}
\usepackage{wrapfig}
\usepackage{sidecap}
\usepackage{dcolumn}%
\usepackage{fullpage}
\def\Bb{\mathcal{B}}

\def\Mm{\mathcal{M}}

\def\Pp{\mathcal{P}}

\def\Xx{\mathcal{X}}

\def\N{\mathbb{N}}
\def\P{\mathbb{P}}
\def\R{\mathbb{R}}

\def\EXP#1{e^{#1}}

\def\EXPECT{{\mathbb{E}}}

\def\RELENT#1#2{\Rr\left(#1|#2\right)}

\def\VAR{\mathrm{Var}}

\def\COMMA{\,,}             %
\def\PERIOD{\,.}            %
\def\SEP{{\,|\,}}           %

\def\VIZ#1{(\ref{#1})}      %

\def\BIGO{\mathcal{O}}

\newtheorem{remark}{Remark}[section]
\newtheorem{condition}{Condition}[section]

\def\MB{\mathcal{M}_b}   %

\def\CGENF{\Lambda}
\def\CCGENF{\tilde\Lambda}
\def\ESET{\mathcal{E}}

\def\PP{{P^{\theta}}}

\def\QQ{{P^{\theta+\epsilon}}}

\def\PPT{{P_{[0,T]}}}
\def\PPTE{{P_{[0,T]}^\theta}}

\def\QQT{{Q_{[0,T]}}}
\def\XXT{X}

\def\DRIFTA{a}
\def\DRIFTB{b}

\def\PPTA{P_{[0,T]}}
\def\PPTB{Q_{[0,T]}}
\def\PPTAA{P}
\def\PPTBB{Q}

\def\XIQPp{\Xi_+(Q\SEP P;f)}
\def\XIQPm{\Xi_-(Q\SEP P;f)}
\def\XIQPpm{\Xi_\pm(Q\SEP P;f)}
\def\XIQPpT{\Xi_+(\QQT\SEP \PPT;\FOBS)}
\def\XIQPmT{\Xi_-(\QQT\SEP \PPT;\FOBS)}
\def\XIQPpmT{\Xi_\pm(\QQT\SEP \PPT;\FOBS)}

\def\PROCA{X_t}
\def\PROCAI{X_0}
\def\PROCAS{X_s}
\def\PROCB{Y_t}
\def\PROCBI{Y_0}

\def\FOBS{\mathcal{F}}

\def\ast{\{\widetilde\sigma_t\}_{t\ge 0}}

\def\EQUIL{\mu_{\sigma}}

\def\RELENTR{\mathcal{R}}
\def\RELENT#1#2{\mathcal{R}\left({#1}\SEP{#2}\right)}
\def\ENTRATE#1#2{\mathcal{H}({#1}\SEP{#2})}

\def\RELENTR#1#2{\mathcal{H}\left({#1}\SEP{#2}\right)}
\def\RELENT#1#2{\mathcal{R}\left({#1}\SEP{#2}\right)}
\def\RELENTF#1#2{\mathcal{R}_\phi\left({#1}\SEP{#2}\right)}
\def\CHISQ#1#2{\chi^2\left({#1}\SEP{#2}\right)}
\def\FISHERR#1{\mathcal{I}_{\mathcal{H}}({#1})}
\def\FISHER#1{\mathcal{I}({#1})}
\def\TVNORM#1#2{\|{#1}-{#2}\|_{\mathrm{TV}}}

\def\SEP{{\,||\,}}           %

\def\var{{\mathrm{Var}}}
\def\cov{{\mathrm{Cov}}}

\def\EXPECT{{\mathbb{E}}}
\def\RELENTR#1#2{\mathcal{H}\left({#1}\SEP{#2}\right)}
\def\EQUIL{\mu^\theta}

\def\EQUIL{\mu^\theta}

\def\ast{\{\widetilde\sigma_t\}_{t\ge 0}}

\def\ast{\{\widetilde\sigma_t\}_{t\ge 0}}

\title{%
Path-space information bounds for uncertainty quantification and sensitivity analysis of stochastic dynamics
\thanks{%
This material is based upon work supported by the U.S. Department of Energy Office of Science, 
Office of Advanced Scientific Computing Research, Applied Mathematics program under 
Award Number DE-SC-0010539 (P.D.), DE-SC-0010723 (M.A.K, Y.P.), 
DE-SC-0010549 (P.P.).}
}

\author{%
Paul Dupuis\thanks{%
Division of Applied Mathematics, Brown University,
Providence, RI 02912, USA,
{\tt dupuis@dam.brown.edu}. Research supported in part by the National Science Foundation (DMS-1317199).}
\and
Markos A. Katsoulakis\thanks{%
Department of Mathematics and Statistics, University of Massachusetts,
Amherst, MA 01003--9305, USA,
{\tt markos@math.umass.edu}}
\and
Yannis Pantazis\thanks{%
Department of Mathematics and Statistics, University of Massachusetts,
Amherst, MA 01003--9305, USA,
{\tt pantazis@math.umass.edu}}
\and
Petr Plech\'a\v{c}\thanks{%
Department of Mathematical Sciences, University of Delaware, Newark, DE 19716, USA,
{\tt plechac@math.udel.edu}}
}
\begin{document}
\maketitle
\begin{abstract}
Uncertainty quantification  is a primary challenge for reliable modeling and simulation of complex  stochastic dynamics.
Such problems  are typically plagued with incomplete information  that may
enter  as  uncertainty in the  model parameters, or even in the model itself. Furthermore,  due to their dynamic nature, 
we need to assess the impact of these uncertainties on the transient and long-time  behavior of the stochastic models
and derive corresponding uncertainty bounds for observables of interest.  A special class of such challenges is parametric
uncertainties in the model and in particular sensitivity analysis along with the corresponding sensitivity bounds for
stochastic dynamics. Moreover, sensitivity analysis can be further complicated in models with a high number of parameters
that render straightforward approaches, such as gradient methods, impractical.
In this paper, we derive uncertainty and sensitivity bounds for path-space observables of stochastic dynamics  in terms
of new  goal-oriented divergences; the latter  incorporate  both  observables and  information theory objects such as the relative entropy rate. 
These bounds are tight, depend on the  variance  of the particular observable and  are computable through
Monte Carlo  simulation. In the case of sensitivity analysis, the derived sensitivity bounds rely on the path Fisher Information Matrix,
hence they depend only on local dynamics and are  gradient-free. These features  allow for  computationally efficient implementation in systems with a high
number of parameters, e.g.,  complex reaction networks and  molecular simulations.

Version: \today
\end{abstract}
\section{Introduction}\label{intro}

In this paper, we derive uncertainty and sensitivity bounds for path-space observables of stochastic dynamics
in terms of suitable information theoretic divergences such as relative entropy rate (RER) and path-space Fisher Information Matrix (pFIM). 
Reliable modeling and simulation of complex systems often suffers from incomplete information that may 
enter as  uncertainty in the  model parameters, or even in the model itself. 
Here we develop an approach that
provides uncertainty bounds for observables of interest in the transient and long-time  behavior of the stochastic models.
The bounds are expressed   in terms
of a new  goal-oriented divergence that  incorporates   observables,
 as well as    path-space information theory objects such as the relative entropy rate. 
The presented method also yields bounds on parametric sensitivity for stochastic dynamics, e.g., for solutions to stochastic differential equations. It is particularly
useful in realistic stochastic models, for example, biochemical reaction networks, which are characterized by 
a high number of parameters that render classic sensitivity analysis approaches, such as gradient methods, impractical.
We present sensitivity bounds that are computable and sufficiently sharp. 

Estimating sensitivity indices appears as a common task in many applications ranging
from engineering and financial mathematics to biochemistry. Methods that apply Monte Carlo
simulations to estimate the gradients directly include finite-difference approximations combined
with coupling methods \cite{Rathinam:10, Anderson:12, AK:2013}, likelihood ratio and Girsanov methods \cite{Glynn:90, Plyasunov:07},
polynomial chaos expansions \cite{Kim:07}, path-wise methods \cite{Khammash:12}, linear response \cite{Hairer:10}, etc. 
In another
direction, information-based  sensitivity analysis approaches have been proposed as means to  quantify the overall
behavior of the system  and not just the response of a specific observable function, \cite{Majda, Komorowski:11, Pantazis:Kats:13}.
These sensitivity analysis methods employ information theory metrics such as the relative
entropy (also known as the Kullback-Leibler divergence)
as well as the Fisher Information Matrix (FIM). Moreover, taking into account  that the
 stationary distribution is rarely known in complex stochastic dynamics,
these information-based methods resort either to linearized Gaussian approximations of the underlying process
\cite{Komorowski:11}, or they rely on  path-space objects such as the relative entropy rate and the path Fisher Information Matrix, \cite{Pantazis:Kats:13, PKV:2013}.
The latter approach is exact since no approximation is necessary. It is also gradient-free in the sense that simulation for a single model (parameter)  yields bounds for all parameter perturbations.%

Overall, gradient-free sensitivity analysis methods such as the ones based on pFIM, 
\cite{Majda,Komorowski:11,Pantazis:Kats:13,PKV:2013} are highly appropriate for systems 
with a high-dimensional parameter space since they allow for an efficient exploration of the 
parameter space without the calculation of a very high number of directional derivatives. 
In the stochastic dynamics setting, the bounds we present avoid expensive Monte Carlo simulations of sensitivity
indices by providing error bounds for them. The derived bounds are based on the path-space FIM 
and are obtained from different inequalities and representations of relative entropy. It is also desirable to provide bounds based on  Fisher
information because the (static) FIM  is a tool extensively utilized in optimal experimental
design, as well as in statistics, for estimation, identifiability, etc. Moreover, in order to
obtain the tightest possible bounds, it is crucial to find the optimal constant that multiplies
the Fisher information in these inequalities.

The presented  results rely in part on an upper bound derived recently in  \cite{Chowdhary:13}
and a companion lower bound in \cite{Li:12}, for functionals of probability measures
$P\in \Pp(\Omega)$ and $Q\in \Pp(\Omega)$, where $Q$ is viewed %
as the ``true'' probabilistic model, and $P$ is a computationally tractable ``nominal'' or
``reference'' model, e.g., a surrogate model. In this paper we start our analysis by showing  that these inequalities,
for bounded observables $f$ of random variables with probabilities $P$ and $Q$, can be 
rewritten in the form 
\begin{equation}\label{lower:upper:bound:-1}
 \Xi_{-}(Q, P; f)
\leq \EXPECT_{Q}[f] - \EXPECT_{P}[f] \leq
 \Xi_+(Q, P; f)\, ,
\end{equation}
where $\Xi_{+}(Q, P; f)\ge 0$ ($\Xi_{-}(Q, P; f) \le 0$), and  $\Xi_{\pm}(Q, P; f)=0$ if and only if $P=Q$ or $f$
is deterministic a.s. with respect to $P$. Due to these properties,  $\Xi_{+}(Q, P; f)\ge 0$ (and $-\Xi_{-}(Q, P; f)$) is a \textit{goal-oriented
divergence},  incorporating in the definition  the observable $f$. Furthermore,  $\Xi_{\pm}(Q, P; f)$ depend on the relative entropy of $Q$ with respect to $P$ and
it admits an explicit representation (see Theorem~\ref{representation}).
We view these weak error bounds (i.e., errors in averages or expected values
for various classes of functions) as Uncertainty Quantification (UQ) bounds for the observables of interest $f$. 
Furthermore, the UQ bounds \VIZ{lower:upper:bound:-1}  characterize the errors incurred if one uses the more computationally
tractable $\EXPECT_P[f]$ instead of $\EXPECT_Q[f]$.
As it is discussed in Section~\ref{SecUQ}, UQ bounds of the type \VIZ{lower:upper:bound:-1}  can be derived from
different divergences used to discriminate between two probability measures $P$ and $Q$.
For example, a common choice is based on the Csisz\'ar-Kullback-Pinsker (CKP) inequality, which bounds the
total variation norm by the relative entropy.
Another approach uses $\chi^2$-divergence (or Pearson divergence) and derives a bound by
a direct application of Cauchy-Schwarz inequality.
The bounds \VIZ{lower:upper:bound:-1} presented in this paper are based on the variational
characterization of relative entropy used in \cite{Chowdhary:13}. The variational approach guarantees optimal
constants in the estimates and thus tighter bounds.

In the context of parametrized models the general UQ bounds \VIZ{lower:upper:bound:-1} give a tool for estimating
sensitivity of observables to perturbations in model parameters. More precisely, given a parametric family of
probability measures $\PP(d\omega)$, $\theta\in\R^k$, on the common measurable space $(\Omega, \Bb)$,
we study bounds on perturbations of $\EXPECT_{\PP}[f]$ under changes of $\theta$. The bounds on sensitivity
indices for parametric model families $\PP$ then follow by asymptotic expansions of $Q\equiv P^{\theta + \epsilon v}$
in $\epsilon$, which is a straightforward procedure under smoothness assumptions when the parameter is finite
dimensional, i.e., $\theta,v\in\R^k$ and $|v|=1$. The derived sensitivity bounds can be viewed as sharp and computable bounds
for the weak error of bounded and continuous functions in cases when the
measure $\PP$ is approximated by $P^{\theta+\epsilon v}$, under assumptions of smoothness on the mapping  $\theta\mapsto\PP$.
The mapping defines a finite dimensional submanifold parametrized by $\theta\in\R^k$ of the manifold of probability
measures $\Pp(\Omega)$ on $\Omega$. For the sensitivity indices defined by
$
S_{f,v}(\PP) = \lim_{\epsilon\rightarrow\infty} \frac{1}{\epsilon}\left(\EXPECT_{P^{\theta + \epsilon v}}[f] - \EXPECT_\PP[f]\right)\COMMA
$
we establish estimates of the type  
\begin{equation}\label{SB:1:intro}
|S_{f,v}(\PP)| \leq \sqrt{\VAR_\PP(f)} \sqrt{v^T\FISHER{\PP}v}\COMMA
\end{equation}
where $\FISHER{\PP}$ is the  FIM of $\PP$. It is worth noting the decomposition
of the right hand side of the above sensitivity bound into the product of two terms, with each term capturing different
aspects of the sensitivities.

A primary novelty of the presented results is their application to cases where the model is represented by a path
measure for a Markov process. Thus the proposed UQ and sensitivity bounds are also applicable for the weak
error of path-dependent quantities.
With stochastic dynamics in mind, we consider a stochastic process $\{X_t\}_{t\ge 0}$  with 
the stationary measure $\mu(dx)$, and a process $\{Y_t\}_{t\ge 0}$ with the initial measure
$\nu(dx)$, and we denote by $P=\PPT$, $Q=\QQT$ the respective measures on the path space.
As previously, $\QQT$ is viewed as the ``true'' measure while $\PPT$ as the ``nominal'' model.
We consider as an observable a measurable functional $\FOBS(\{X_t\}_{0\leq t \leq T})$ of the
process.
The derived UQ bounds are now set on path space and characterize the errors incurred when
approximating $\EXPECT_{\QQT}[\FOBS]$ by $\EXPECT_{\PPT}[\FOBS]$
\begin{equation}\label{lower:upper:bound:path:-1}
 \Xi_{-}(\QQT, \PPT; \FOBS) \leq \EXPECT_{\QQT}[\FOBS] - \EXPECT_{\PPT}[\FOBS]
 \leq \Xi_+(\QQT, \PPT; \FOBS)\PERIOD
\end{equation}
Even though the path UQ bound in \VIZ{lower:upper:bound:path:-1} is a direct consequence of \VIZ{lower:upper:bound:-1}
it can be further elaborated using  properties and asymptotics  of the relative entropy between path distributions, such as 
the relative entropy rate (RER), denoted by $\ENTRATE{Q}{P}$, which measures the information loss per unit time (for a definition of RER see \VIZ{rer:def}). %
The RER
for large classes of stochastic dynamics 
 is a computable quantity, \cite{Pantazis:Kats:13},  implying in turn that the bounds in \VIZ{lower:upper:bound:path:-1} are
computable using Monte Carlo simulation. 
Furthermore, in a calculation reminiscent of the G\"{a}rtner-Ellis Theorem, we show that %
 the bounds  (\ref{lower:upper:bound:path:-1}) in 
the $T \to \infty$ limit take the form
\begin{equation}
\label{repr:formula:GE:intro}
\Xi_{\pm}(Q\SEP P; \FOBS)=G^{\pm}_{P,\FOBS}\Big(\ENTRATE{Q}{P}\Big)\, ,
\end{equation}	
where $G^{\pm}_{P,\FOBS}(0)=0$ and the function $G^{\pm}_{P,\FOBS}$ is calculated in terms of the cumulant generating function of the observable $\FOBS$ under the model $P$.
Finally, (\ref{repr:formula:GE:intro}) demonstrates the key  role played by the relative entropy rate  $\ENTRATE{Q}{P}$ for  uncertainty quantification  of stochastic processes and in general for models with correlated data. These bounds  further  justify the sensitivity analysis based on relative entropy rate and coarse-graining methods developed in  \cite{Pantazis:Kats:13} and  \cite{Kats:Plechac:13}, respectively.

An implication of the path UQ bounds (\ref{lower:upper:bound:path:-1}) is that
sensitivity analysis bounds which are general and valid in both transient and long-time regimes are possible. 
In particular, when assuming a parametric family of path distributions parametrized by $\theta\in\mathbb R^k$
and for time-averaged observables of the form 
$
\FOBS(\{X_t\}_{0\leq t \leq T}) = \frac{1}{T}\int_0^T f(X_{s})\,ds,
$
we obtain sensitivity bounds such as (\ref{SB:1:intro}) for both  transient and long-time regimes. For example, for the   stationary distribution $\mu^\theta$ (unknown for most stochastic dynamics models) we have the bound 
\begin{equation}\label{SB:2:intro}
|S_{f,v}(\mu^\theta)| \leq \sqrt{\tau(f)} \sqrt{v^T \FISHERR{P^\theta} v}\COMMA
\end{equation}
where $\tau(f)$ is the integrated autocorrelation time (IAT). In Monte Carlo simulation, the calculation of IAT is a necessary step since it provides the variance of the estimated observable, $\FOBS$, \cite{Liu:MC}.  Furthermore,  $\FISHERR{P^\theta}$
is the path FIM which corresponds to the Hessian of the RER. The path FIM is also computable for large classes of stochastic dynamics; for example,  for chemical reaction networks the path FIM is a sparse,  block-diagonal matrix, hence all related computations   scale linearly with the dimension  of the parameter vector $\theta$, \cite{PKV:2013}. Therefore, the path FIM is computationally feasible, even for systems with a very high-dimensional parameter space.
For completeness in the presentation, we refer to Appendix~\ref{appendix:sec}
for  the RER and the path FIM formulas for various classes of  Markov processes.

We present %
several examples of the derived sensitivity bounds and their tightness  %
is demonstrated. 
In particular,  for the exponential family of distributions, the sensitivity bound becomes an equality, showing the sharpness of the bounds.
Additionally, we compare the ``static'' and the path-space sensitivity bounds for simple Markov processes where
the stationary distribution is explicitly known.
We note though that for non-equilibrium steady state systems the stationary distribution is generally not known,
therefore, comparisons are not feasible and only the path-space sensitivity bound \VIZ{SB:2:intro} can be computed. 
\section{Uncertainty quantification information inequalities and sensitivity bounds}\label{SecUQ}

\subsection{Distances and divergences of probability measures}\label{prob:dist:div}
Bounds of the type \VIZ{lower:upper:bound:-1} are based on characterizing a distance or divergence between the 
measures, $Q$, $P$, under which the averages are evaluated.
While our primary goal is to characterize the bounds based on relative entropy, other divergences can be also
used to derive similar bounds with different levels of sharpness.

\begin{definition}
The total variation norm between two probability measures $Q$ and $P$  on $(\Omega,\Bb)$ is defined by
\begin{equation}
\TVNORM{Q}{P} = \sup_{A\in\Bb} |Q(A)-P(A)|\PERIOD
\end{equation}
\end{definition}

We also consider two pseudo-distances, or divergences in the statistics terminology. 
\begin{definition}
For two probability measures $Q$, $P$
on $(\Omega,\Bb)$ the relative entropy (information divergence,
Kullback-Leibler divergence) of $Q$ with respect to $P$ is defined by
\begin{equation}
\RELENT{Q}{P} = \begin{cases}
                  \int \log\frac{dQ}{dP}(\omega)\, Q(d\omega)=\int \frac{dQ}{dP}(\omega) 
                   \log\frac{dQ}{dP}(\omega)\, P(d\omega) \COMMA & \mbox{if $Q\ll  P$ and $\frac{dQ}{dP} \log\frac{dQ}{dP}$ is $P$-integrable,}\\
                  +\infty & \mbox{otherwise.}
                \end{cases}
\end{equation}
\end{definition}
The Kullback-Leibler divergence is a particular case of a family of Csisz\'ar $\phi$-divergences which are 
functionals of the form 
\begin{equation}
\RELENTF{Q}{P}=\begin{cases}
\int \phi\left(\frac{dQ}{dP}(\omega)\right)\, P(d\omega)\COMMA & \mbox{if $Q\ll  P$ and $\phi\left(\frac{dQ}{dP}\right)$ is $P$-integrable,}\\
+\infty & \mbox{otherwise,}
\end{cases}
\end{equation}
for a convex function $\phi:\R^+\to\R$ with $\phi(1)=0$. In the case
of the relative entropy we have $\phi(x) = x  \log x$. Another choice of the convex function, $\phi(x) = (x-1)^2$, gives a member of the $\phi$-divergence family known as $\chi^2$-divergence.
\begin{definition}
The $\chi^2$-divergence of two probability measures $Q$, $P$ on $(\Omega,\Bb)$ is defined by
\begin{equation}\label{chi}
\CHISQ{Q}{P} =\begin{cases}
                  \int \left(\frac{dQ}{dP}(\omega)-1\right)^2 \,P(d\omega)\COMMA & \mbox{if $Q\ll  P$,}\\
                   +\infty & \mbox{otherwise.}
              \end{cases}
\end{equation}
\end{definition}

\subsection{Information inequalities and goal-oriented divergence}

We turn to a variational formulation that provides sharp weak error estimates in terms of relative entropy.
Let $\Mm(\Omega)$ denote the measurable functions from $\Omega$ into $\mathbb{R}$ and let $\MB(\Omega)$ be the subset of functions 
that are uniformly bounded. For $f \in \MB(\Omega)$ and $c\in\R$ we introduce the {\it cumulant generating function} 
(logarithmic moment generating function)
\begin{equation}\label{cgen:def}
 \CGENF_{P,f}(c) = \log \EXPECT_P\big[\EXP{cf}\big] \equiv \log \int \EXP{c f} \,dP  \PERIOD
\end{equation}
We restrict our analysis to the functions $f$ for which $\CGENF_{P,f}(c)$ is finite at least in a neighborhood
of the origin. More specifically, we have the following definition of the set $\ESET$.
\begin{definition}
A function $f \in \MB(\Omega)$ belongs to the set $\ESET$ if and only if there exists $c_0>0$ such that $\CGENF_{P,f}(\pm c_0)<\infty$.
\end{definition}

The properties of $\CGENF_{P,f}$
then guarantee that $\CGENF_{P,f}(c)$ is finite for all $c\in [-c_0,c_0]$. %
We note that $\EXPECT_P[|f|]$ is finite for all $f\in \ESET$.
It will be more convenient to work 
with the cumulant generating function of the centered observable $\tilde f \equiv f - \EXPECT_P[f]$:
\begin{equation}\label{ccgen:def}
 \CCGENF_{P,f}(c) = \log \EXPECT_P\big[\EXP{c(f-\EXPECT_P[f])}\big] \equiv \log \int \EXP{c (f - \int f\,dP)} \,dP  \PERIOD
\end{equation}

Recalling the basic properties of the cumulant generating function for $f\in\ESET$ that is not essentially constant,
we have that $\CCGENF_{P,f}(\cdot)$ is  a strictly convex function which is $C^\infty$ in a neighborhood of the origin,
with the derivatives $\CCGENF_{P,f}^{(k)}(0)$ defining the cumulants of $f-\EXPECT_P[f]$ under $P$. In particular,
$\CCGENF_{P,f}(0) = \CCGENF'_{P,f}(0) = 0$ and $\CCGENF''_{P,f}(0) = \VAR_P(f)$.
The following characterization of exponential integrals is well-known in statistics and large deviation theory
(see e.g., \cite{Dupuis:97}). For the sake of completeness we present it here together with a proof.
\begin{lemma}\label{variational-cumulant}
Let $f \in \MB(\Omega)$ and $P$ be a probability measure on $(\Omega,\Bb)$. Then
\begin{equation}\label{var_cumulant}
 \log\EXPECT_P\left[\EXP{f}\right] = \sup_{Q\ll P}\left\{ \EXPECT_Q[f] - \RELENT{Q}{P}\right\}\PERIOD
\end{equation}
\end{lemma}
\begin{proof}
It suffices to consider only $Q$ such that $\RELENT{Q}{P}<\infty$ in
(\ref{var_cumulant}). Let the probability measure $R$ be defined by
$dR/dP=\EXP{f}/\mathbb{E}_{P}[\EXP{f}]$. If $\RELENT{Q}{P}<\infty$, then $Q\ll P$
implies $Q\ll R$. Thus%
\begin{align*}
-\RELENT{Q}{P}+\mathbb{E}_{Q}[f]  & =-\mathbb{E}_{Q}\left[  \log\left(
                                      \frac{dQ}{dP}\right)  \right]  +\mathbb{E}_{Q}[f]\\
                                  & =-\mathbb{E}_{Q}\left[  \log\left(  \frac{dQ}{dR}\right)  \right]
                                     -\mathbb{E}_{Q}\left[  \log\left(  \frac{dR}{dP}\right)  \right]
                                     +\mathbb{E}_{Q}[f]\\
                                  & =-\RELENT{Q}{R}+\log\mathbb{E}_{P}[\EXP{f}].
\end{align*}
Now use that $\RELENT{Q}{R}\geq0$ and $\RELENT{Q}{R}=0$ if and only if $Q=R$ \cite[Lemma 1.4.1]{Dupuis:97}.
This establishes (\ref{var_cumulant}) and also shows that $R$ is the
supremizing measure.
\end{proof}

By changing $f$ to $c (f - \EXPECT_P[f])$, we obtain a variational formula for the cumulant generating function
\begin{equation}\label{varccgen}
  \CCGENF_{P,f}(c) = \sup_{Q\ll P}\left\{c(\EXPECT_Q[f] - \EXPECT_P[f]) - \RELENT{Q}{P}\right\}\PERIOD
\end{equation}
The variational characterization gives us the following upper and lower bounds for $f\in \MB(\Omega)$ and $c>0$:
\begin{eqnarray}\label{firstub}
\EXPECT_Q[f] - \EXPECT_P[f] &\leq& \frac{1}{c} \log \EXPECT_P[\EXP{c(f-\EXPECT_P[f])}] + \frac{1}{c}\RELENT{Q}{P}\COMMA\\\label{firstlb}
\EXPECT_Q[f] - \EXPECT_P[f] &\geq& -\frac{1}{c} \log \EXPECT_P[\EXP{-c(f-\EXPECT_P[f])}] - \frac{1}{c}\RELENT{Q}{P}\PERIOD
\end{eqnarray}
These inequalities can be extended to any $f \in \ESET$, and we give the argument for the case of the upper bound (\ref{firstub}).
Recall that $f \in \ESET$ implies $\EXPECT_P[|f|]<\infty$. 
If $\EXPECT_P[\EXP{c(f-\EXPECT_P[f])}]=\infty$, then (\ref{firstub}) holds automatically.
If $\EXPECT_P[\EXP{c(f-\EXPECT_P[f])}]<\infty$, let $f^{a,b}=[f\vee (-a)]\wedge b$ for $a,b \in \R$,
and apply (\ref{firstub}) with $f-\EXPECT_P[f]$ replaced by $f^{a,b}-\EXPECT_P[f]$.
First let $a\rightarrow \infty$ and use the Monotone Convergence Theorem, and then send $b\rightarrow \infty$ 
and use the dominating function $\EXP{c(f-\EXPECT_P[f])}$ to obtain (\ref{firstub}) as written.

Using these inequalities, tight estimates  as in Chowdhary and Dupuis, \cite{Chowdhary:13},
and Li and Xie, \cite{Li:12} can be obtained by optimizing over $c>0$
\begin{equation}\label{lower:upper:bound:1}
\sup_{c>0}\left\{-\frac{1}{c}\CCGENF_{P,f}(-c) - \frac{1}{c}\RELENT{Q}{P}\right\}
\leq \EXPECT_{Q}[f] - \EXPECT_P[f] \leq
\inf_{c>0}\left\{\frac{1}{c}\CCGENF_{P,f}(c) + \frac{1}{c}\RELENT{Q}{P}\right\}\PERIOD
\end{equation}

We refer to upper and lower bounds of this form as  \textit{Uncertainty Quantification Information Inequalities} (UQII). The corresponding bounds  
define a new type of divergence between probability measures $P$ and $Q$ as well as  the observable $f$, hence we refer to it as a {\em Goal-oriented Divergence}. 
More precisely, based on \VIZ{lower:upper:bound:1} we give the following definitions.

\begin{definition}
For any two probability measures $P$ and $Q$ with $\RELENT{Q}{P} <\infty$ and any observable $f\in\ESET$, we define %
\begin{equation}\label{UQII+}
\XIQPp = \inf_{c>0}\left\{\frac{1}{c}\CCGENF_{P,f}(c) + \frac{1}{c}\RELENT{Q}{P}\right\}\, , 
\end{equation}
and similarly
\begin{equation}\label{UQII-}
\XIQPm = \sup_{c>0}\left\{-\frac{1}{c}\CCGENF_{P,f}(-c) - \frac{1}{c}\RELENT{Q}{P}\right\}\, .
\end{equation}
\end{definition}
Then the UQIIs \VIZ{lower:upper:bound:1} are rewritten as
\begin{equation}\label{lower:upper:bound:1b}
\XIQPm
\leq \EXPECT_{Q}[f] - \EXPECT_{P}[f] \leq
\XIQPp\, .
\end{equation}
We next show that $\XIQPp$ and $-\XIQPm$ have the properties of a divergence similar to  the relative entropy and the 
$\chi^2$-divergence. However, the new goal-oriented divergence additionally captures the role of fluctuations of the observable $f$, as is further quantified 
in Theorem~\ref{divergence} and Theorem~\ref{linearization} below. %
More specifically we have:

\begin{theorem}[Goal-oriented Divergence]\label{divergence} 
	Assume that $f\in\ESET$ and $\RELENT{Q}{P}<\infty$.
	Then
	\begin{description}
		\item[{\rm (i)}] $\XIQPp \ge 0$ and $\XIQPm \le 0$,
		\item[{\rm (ii)}] $\XIQPpm=0$ if and only if $Q=P$ or $f$ is constant $P$-a.s.
	\end{description}
\end{theorem}
\begin{proof} 
	{\rm 
		The proofs for $\Xi_{+}$ and $\Xi_{-}$ are similar and therefore we prove only the former case.

		\noindent (i) The proof uses the fact that both terms in the variational definition of $\Xi_+$,
		$$
		\XIQPp = \inf_{c>0} \left\{ \frac{1}{c} \CCGENF_{P,f}(c) + \frac{1}{c}\RELENT{Q}{P}\right\}\COMMA
		$$
		are non-negative. The relative entropy $\RELENT{Q}{P}$ is a divergence hence always non-negative, and thus
		$\frac{1}{c}\RELENT{Q}{P}\ge0$ for all $c>0$. Furthermore, by Jensen's inequality,
		\begin{equation*}
		\frac{1}{c}\CCGENF_{P,f}(c) \equiv \frac{1}{c}\log \EXPECT_{P}\left[e^{c(f-\EXPECT_{P}[f])}\right]
		\ge \frac{1}{c}\log e^{\EXPECT_{P}[c(f-\EXPECT_{P}[f])] } = \EXPECT_{P}[f-\EXPECT_{P}[f]] = 0\PERIOD
		\end{equation*}

		\noindent (ii) If $f=\EXPECT_{P}[f]$ then $\CCGENF_{P,f}(c)\equiv 0$. Since $\RELENT{Q}{P}\in [0,\infty)$,
		\begin{equation*}
		\XIQPp = \inf_{c>0}\left\{\frac{1}{c}\RELENT{Q}{P}\right\} = 0 \PERIOD
		\end{equation*}
		If $Q=P$ then $\RELENT{Q}{P}=0$ and
		\begin{equation*}
		0 \le \XIQPp = \inf_{c>0}\left\{\frac{1}{c}\CCGENF_{P,f}(c)\right\}
		\le \lim_{c\to 0}\frac{1}{c}\CCGENF_{P,f}(c) = \CCGENF_{P,f}'(0) = 0\PERIOD
		\end{equation*}

		\noindent For the reverse direction we can assume $\RELENT{Q}{P} > 0$, since if $\RELENT{Q}{P} = 0$ the conclusion is automatic.
		In this case the infimum must be obtained in the limit $c \rightarrow \infty$, so that $\lim_{c\to\infty}\CCGENF_{P,f}(c)/c = 0$. 
		We claim that $\CCGENF_{P,f}(c)=0$ for all $c \in [0,\infty)$.
		Since $\CCGENF_{P,f}(0)=0$, if $\CCGENF_{P,f}(\hat{c})>0$ for some $\hat{c} \in (0,\infty)$ then $\CCGENF_{P,f}'(\bar{c})>0$ 
		for some $\bar{c} \in (0,\hat{c}]$. Convexity then implies $\liminf_{c\to\infty}\CCGENF_{P,f}(c)/c \geq \CCGENF_{P,f}'(\bar{c})>0$,
		and this contradiction establishes $\CCGENF_{P,f}(c)=0$ for all $c \in [0,\infty)$.
		Since $f\in\ESET$ implies $\CCGENF_{P,f}(c)$ is twice continuously differentiable at $c=0$, 
		$\EXPECT_P\left[(f - \EXPECT_P[f])\right]^{2} = \CCGENF_{P,f}''(0)= 0$, and therefore $f=\EXPECT_P[f]$ $P$-a.s. 
	}
\end{proof}

Furthermore, we derive an analytic formula for the divergences $\XIQPpm$:
\begin{theorem}[Representation]\label{representation} 
	If $f\in\ESET$ with $f\not =\EXPECT_P[f]$ $P$-a.s.\  and $\RELENT{Q}{P} <\infty$ then we have
	\begin{equation}\label{repr:formula}
		\XIQPp = \CCGENF_{P,f}'\big(\Phi^{-1}(\RELENT{Q}{P}) \big)\, \quad \mbox{and}\quad   
	  \XIQPm = \CCGENF_{P,f}'\big(-\Phi^{-1}(\RELENT{Q}{P}) \big),
	\end{equation}
	where $$\Phi(c):=-\CCGENF_{P,f}(c)+c\CCGENF_{P,f}'(c)$$ is a strictly increasing function on $(0,\bar c)$, and 
where $\bar c = \sup \{c:\CCGENF_{P,f}(c)<\infty\}$. %
\end{theorem}
\begin{proof}
	{\rm 
		Let $\Theta_+(c;\rho)\equiv\frac{1}{c}\CCGENF_{P,f}(c) + \frac{1}{c}\rho^2$, where $\rho^2=\RELENT{Q}{P}$.
		Then 
		\begin{equation}\label{minimization}
		\XIQPp = \inf_{c>0}\Theta_+(c;\rho)\PERIOD
		\end{equation}
		We use that $\CCGENF_{P,f}(0) = \CCGENF'_{P,f}(0) = 0$, and that $f\not =\EXPECT_P[f]$ $P$-a.s.\ implies $\CCGENF_{P,f}$ is strictly convex.
		If $\rho \not = 0$ then
		$\Theta_+(c;\rho)$ tends to $\infty$ as $c \downarrow 0$ and as $c\uparrow \infty$. 
		Hence the infimum is achieved. 
		Suppose an infimum of $A>0$ is achieved at two points $0<c_1<c_2<\infty$, so that $\CCGENF_{P,f}(c_i)+\rho^2=c_iA$, $i=1,2$.
		If $\bar c = (c_1+c_2)/2$, then the strict convexity of $\CCGENF_{P,f}$ implies $\CCGENF_{P,f}(\bar c)+\rho^2< \bar c A$. 
		This contradicts the minimality of $c_i$, and thus shows the minimizer is unique.
		Since $\CCGENF_{P,f}(0) = \CCGENF'_{P,f}(0) = 0$ we can continuously extend the function $\Theta_+(c,0)$ to
		$c=0$ by $\Theta_+(0,0)=0$. Then by direct calculation and lower semicontinuity the optimization problem in (\ref{repr:formula}) extended to $c\geq 0$
                has the unique minimizer $c^*(0) = 0$ 
		with the minimum value equal to $0$. Then $\inf_{c\geq0}\Theta_+(c;\rho)$ is well defined and achieves the infimum for all $\rho\in\R$.
		Since
		$\CCGENF_{P,f}(\cdot)$ is a proper convex function and $C^{\infty}$ in its domain of finiteness we have, for all $\rho\in\R$, 
		the optimality condition
		\begin{equation}\label{optimality}
		-\frac{1}{c^2}\CCGENF_{P,f}(c) + \frac{1}{c} \CCGENF'_{P,f}(c) - \frac{1}{c^2}\rho^2 = 0\PERIOD
		\end{equation}
		Multiplying \VIZ{optimality} by $c^2$, we obtain that the minimizer $c^*=c^*(\rho)$ satisfies
		\begin{equation}\label{optimality1a}
		   -\CCGENF_{P,f}(c) +  c\CCGENF'_{P,f}(c) = \rho^2 \PERIOD
		\end{equation}
		We will use that $\CCGENF_{P,f}(c)$ is a log moment generating function with $\CCGENF_{P,f}(c)<\infty$ for $c$ in an open neighborhood of 
		zero and $\CCGENF_{P,f}'(0)=0$. These imply that if $\CCGENF^*(t)$ is the Legendre-Fenchel transform of $\CCGENF_{P,f}$, i.e.,
		$\CCGENF^*(t) = \sup_{c>0}\{c t -\CCGENF_{P,f}(c) \}$, then $\CCGENF^*(t)$ has its unique minimum at $t=0$, and $\CCGENF^*(t)\rightarrow \infty$
		as $t\rightarrow \infty$.
		If $t(c)$ is the unique solution of $\CCGENF^*(t)=c$, then it follows from convex duality that 
		$$
		\Phi(c) = -\CCGENF_{P,f}(c)+c\CCGENF_{P,f}'(c)=\CCGENF^*(t(c))
		$$
		is strictly increasing and maps $(0,\bar{c})$ onto $(0,\infty)$. Therefore, from  \VIZ{optimality1a} we have
		\begin{equation}\label{optimal}
		c^*=c^*(\rho)=\Phi^{-1}(\rho^2)\, .
		\end{equation}
		Substituting in \VIZ{minimization} and using \VIZ{optimality1a}, we have that 
		\begin{equation}
		\XIQPp = \Theta_+(c^*(\rho);\rho)=\CCGENF'_{P,f}(c^*(\rho))=\CCGENF_{P,f}'\big(\Phi^{-1}(\rho^2)\big)\, .
		\end{equation}
		The representation of the lower bound $\XIQPm = \CCGENF_{P,f}'\big(-\Phi^{-1}(\rho^2)\big)$ is computed  in a similar way.
}
\end{proof}

From the proof above we deduce that the dependence on the cumulant generating function of $f$ can be removed if a bound is available. 
Note that if $\Psi:\R\to\R$ is convex with a minimum of zero at the origin, then in the definition of $\Psi^*(t)$, its Legendre-Fenchel transform, the supremum can be restricted to $(0,\infty)$.
	
\begin{corollary}
	Let $\Psi:\R\to\R$ be a convex and continuously differentiable function such that $\Psi(0) = \Psi'(0) = 0$ and
	$$
	\CCGENF_{P,f}(c)\equiv \log\EXPECT_P[\EXP{c(f - \EXPECT_P[f])}] \leq \Psi(c)\COMMA
	$$
	and define $\Psi_+^\sharp(t) = (\Psi_+^*)^{-1}(t)$ as the (generalized) inverse of the Legendre-Fenchel
	transform $\Psi^*(t) = \sup_{c>0}\{c t - \Psi(c)\}$ of the function $\Psi$.
	Then
	\begin{equation}\label{relentropy_bound}
	\EXPECT_{Q}[f] - \EXPECT_P[f] \leq \Psi_+^\sharp(\RELENT{Q}{P})\PERIOD
	\end{equation}
\end{corollary}
We end this section by relating the derived bounds to existing information-theoretic inequalities. 
The Csisz\'ar-Kullback-Pinsker inequality states that (for proofs see, e.g., \cite{Tsybakov:08})
\begin{equation}\label{CKP_ineq}
 \TVNORM{Q}{P}  \leq \sqrt{2\RELENT{Q}{P}}\PERIOD
\end{equation}
Using $\TVNORM{Q}{P} = \sup_{\|f\|_\infty \leq 1} \{\EXPECT_Q[f] - \EXPECT_P[f]\}$
and the Csisz\'ar-Kullback-Pinsker inequality \VIZ{CKP_ineq} we obtain 
\begin{equation} \label{CKPbased_bound}
 \left| \EXPECT_Q[f] - \EXPECT_P[f] \right| \leq \|f\|_\infty \sqrt{2\RELENT{Q}{P}}\PERIOD
\end{equation}
The constant in front of the pseudo-distance can be improved by using the $\chi^2$-divergence
instead of the relative entropy. Observing that $\left|\EXPECT_P[f] - \EXPECT_Q[f]\right|=
\left|\int f \,\left(1-\frac{dQ}{dP}\right)\, dP\right|$ and applying 
the Cauchy-Schwarz inequality to the right-hand side we have, for
$P$, $Q$ two probability measures on $(\Omega,\Bb)$ with $Q\ll P$ and $f\in \MB(\Omega)$,
\begin{equation}\label{chisquare_bound}
 \left| \EXPECT_P[f] - \EXPECT_Q[f] \right| \leq \sqrt{\VAR_P(f)} \sqrt{\CHISQ{Q}{P}}\PERIOD
\end{equation}
However, this bound is weaker than the new derived bound derived, \VIZ{lower:upper:bound:1b}, since
in general $\RELENT{Q}{P} \leq \CHISQ{Q}{P}$.

\subsection{Linearization of the UQ bounds} 
The UQ bounds \VIZ{lower:upper:bound:1b} and the representations \VIZ{repr:formula}  can be made more explicit in terms of the asymptotic expansion at $\RELENT{Q}{P}=0$, i.e., when $Q$ is a perturbation $P$. We first prove an asymptotic
expansion for the solution of the optimization problems in \VIZ{lower:upper:bound:1}.
\begin{lemma}\label{optim:asymptotics}
For two probability measures  $P$, $Q$ on $(\Omega,\Bb)$ set $\rho^2 = \RELENT{Q}{P}$.
Assume $\rho^2<\infty$ and that $f\in\ESET$ with $f\not =\EXPECT_P[f]$ $P$-a.s. 
Then there exists a function $c^*(\rho)$  which is the unique solution of 
\begin{eqnarray*}%
&(P_+)\;\;\;\; &\inf_{c>0}\left\{\frac{1}{c}\CCGENF_{P,f}(c) + \frac{1}{c}\RELENT{Q}{P}\right\}\\
\mbox{as well as}&& \\
&(P_-)\;\;\;\;&\sup_{c>0}\left\{-\frac{1}{c}\CCGENF_{P,f}(-c) - \frac{1}{c}\RELENT{Q}{P}\right\}\PERIOD
\end{eqnarray*}
Furthermore, there is $\rho_0>0$ such that the optimal solution $c^*(\rho)$ is $C^\infty$ in $(0,\rho_0)$ and admits the expansion 
\begin{equation}\label{c:star:rho}
c^*(\rho) = c_1^*\rho + \BIGO(\rho^2)\COMMA
\end{equation}
where
\begin{equation}\label{optimalc_rho}
c_1^* = \sqrt{\frac{2}{\VAR_P(f)}}\PERIOD
\end{equation}
\end{lemma}
\begin{proof}
We first solve $(P_+)$. 
Let $\Theta_+(c;\rho)=\frac{1}{c}\CCGENF_{P,f}(c) + \frac{1}{c}\rho^2$. Following Theorem~\ref{representation}
we obtain the optimality condition \VIZ{optimality}.
Multiplying \VIZ{optimality} by $c$, we define
\begin{equation}
\label{optimality2}
G(c, \rho):=-\frac{1}{c}\CCGENF_{P,f}(c) +  \CCGENF'_{P,f}(c) - \frac{1}{c}\rho^2\PERIOD
\end{equation}
Next, we apply the Implicit Function Theorem at $c=0, \rho=0$ as follows; first we have that
$$\frac{\partial}{\partial c}G(c, 0)=\frac{1}{2}\CCGENF_{P,f}''(0)+{\cal O}(c)\, ,
$$
and thus obtain
$$
\lim_{c\to 0}\frac{\partial}{\partial c}G(c, 0)=\frac{1}{2}\CCGENF_{P,f}''(0)=\VAR_P(f)\, .
$$
Since $\VAR_P(f)>0$, by the Implicit Function Theorem there exists  
a unique solution $c^*(\rho)>0$, $c^*(0)=0$ of $G(c,\rho)=0$ (and thus of  \VIZ{optimality})
and $c^*(\rho)\in C^\infty$ for $\rho$ in a neighborhood of the origin. Differentiating $G(c^*(\rho),\rho)=0$ and setting $\rho=0$
yields terms in the Taylor expansion of $c^*(\rho)$. In particular, using the notation $\dot c^* = d c^*/d\rho$, we have 
$c^*(\rho) \CCGENF''_{P,f}(c^*(\rho)) \dot c^*(\rho) =2\rho$,
and thus by setting $\dot c^*(0) = \lim_{\rho\to 0+} \dot c^*(\rho)$
we have $(\dot c^*(0))^2 = 2/\CCGENF''_{P,f}(0)$, which concludes the proof by observing again that $\VAR_P(f)=\CCGENF''_{P,f}(0)$.

To prove that $c^*(\rho)$ is also the solution of $(P_-)$ we observe that
$$
\sup_{c>0}\left\{-\frac{1}{c}\CCGENF_{P,f}(-c) - \frac{1}{c}\RELENT{Q}{P}\right\} = 
-\inf_{c>0}\left\{\frac{1}{c}\CCGENF_{P,f}(-c) + \frac{1}{c}\RELENT{Q}{P}\right\}\COMMA
$$
and using the same arguments as for $(P_+)$ we conclude that the unique solution is obtained as the solution of the optimality
condition
$$
-\frac{1}{c^2}\CCGENF_{P,f}(-c) - \frac{1}{c} \CCGENF'_{P,f}(-c) - \frac{1}{c^2}\rho^2 = 0\COMMA\;\;\; c>0\COMMA
$$
which is, under the change of the variable $c \to -c$,  the same as \VIZ{optimality} and thus analogous calculations yield the result.
\end{proof}

Next, substituting the expansion in $\rho$ for the optimal value \VIZ{optimalc_rho} %
we obtain asymptotics in $\rho^2 = \RELENT{Q}{P}$ of the upper and lower bounds for the UQ  error \VIZ{lower:upper:bound:1b}.
\begin{theorem}[Linearization]\label{linearization}
Under the assumption that $f\in\ESET$ with $f\not =\EXPECT_P[f]$ $P$-a.s., we have
	\begin{description}
		\item[{\rm (i)}]  the asymptotic expansion
		$\XIQPpm=\pm \sqrt{\VAR_P(f)}\sqrt{2\RELENT{Q}{P}} + \BIGO(\RELENT{Q}{P})$, and,
		\item[{\rm (ii)}] an estimate of the weak error%
		\begin{equation}\label{weak_error_bound}
		|\EXPECT_{Q}[f] - \EXPECT_{P}[f]| \leq \sqrt{\VAR_P(f)}\sqrt{2\RELENT{Q}{P}} + \BIGO(\RELENT{Q}{P}).
		\end{equation}
	\end{description}	

\end{theorem}

If needed, the term $\BIGO(\RELENT{Q}{P})$ can be further resolved using the asymptotic expansions  
of $c^*(\rho)$ and $\Theta_\pm(c;\rho)$ defined in Lemma~\ref{optim:asymptotics}, in terms of $\rho^2=\RELENT{Q}{P}$.

\begin{proof}
	The proof follows from the Taylor expansion of \VIZ{repr:formula} in $\rho$, where $\rho^2=\RELENT{Q}{P}$,  around $\rho=0$. 
	First, we note that 
	$\CCGENF_{P,f}(0) = \CCGENF'_{P,f}(0) = 0$ and $\CCGENF''_{P,f}(0)=\VAR_P(f)$.
	Therefore $\Phi^{-1}(0)=0$, and the upper bound becomes
	$$
	\XIQPp = \CCGENF_{P,f}'\big(\Phi^{-1}(\rho^2)\big)=\CCGENF_{P,f}'(0)+\CCGENF_{P,f}''(0)
	                              \Phi^{-1}(\rho^2)+ \BIGO(|\Phi^{-1}(\rho^2)|^2)\PERIOD
	$$
	We conclude using \VIZ{optimal} and the expansion \VIZ{c:star:rho}.
\end{proof}

\subsection{Sensitivity bounds and perturbation analysis}\label{Sensitivity:2}
In this section we consider a smooth parametric family of probability measures $\PP$, $\theta\in \R^k$, and assume that the following (mild) condition.  
\begin{condition}\label{con:smooth_densities}
There is a fixed reference probability measure $R\in\mathcal{P}(\Omega)$ such
that $P^{\theta}\ll R$ for all $\theta\in\mathbb{R}^{k}$. Let $p^{\theta}(\omega)=\frac{dP^{\theta}}{dR}(\omega)$. Then there is a measurable set
$N\subset\Omega$ such that $R(N)=0$, and such that for all $\omega\notin N$ the mapping $\theta\rightarrow p^{\theta}(\omega)$ from $\mathbb{R}^{k}$ to
$(0,\infty)$ is $C^{3}$. Where needed, we also assume the existence of
suitable dominating functions for various functions of $p^{\theta}$.%
\end{condition}

Under Condition~\ref{con:smooth_densities} the relative entropy can be expressed as
$$
\RELENT{P^{\theta + v}}{\PP} = \int p^{\theta+v}(\omega) \log\frac{p^{\theta+v}(\omega)}{p^{\theta}(\omega)} \,R(d\omega).
$$
Using the Taylor expansion and the fact 
$\int [\partial_{\theta_i}  \log p^{\theta}(\omega)] \, p^{\theta}(\omega)\,R(d\omega) = 0$,
we have the perturbative expansion 
\begin{equation}\label{eqn:RE_expansion}
\RELENT{P^{\theta + v}}{\PP} = \frac{1}{2} \sum_{ij} v_i v_j \int \frac{1}{p^{\theta}(\omega)}[\partial_{\theta_i}p^{\theta}(\omega)][\partial_{\theta_j}p^{\theta}(\omega)] \,R(d\omega)+ 
                \BIGO(|v|^3)\PERIOD
\end{equation}
The leading term in this expansion is a quadratic form defined by the FIM
\begin{equation}\label{def:fisher}
\FISHER{\PP}_{ij} \equiv \int \frac{1}{p^{\theta}(\omega)}[\partial_{\theta_i}p^{\theta}(\omega)][\partial_{\theta_j}p^{\theta}(\omega)] \,R(d\omega)
= -\int [\partial^2_{\theta_i\theta_j}\log p^{\theta}(\omega)]\, p^{\theta}(\omega) \,R(d\omega)\PERIOD
\end{equation}
We apply the derived bounds of Theorem~\ref{linearization}  for the weak error in order to obtain bounds on the sensitivity indices (when the derivatives  exist):
\begin{equation}
\label{SI}
S_{f,v}(P^\theta) = \lim_{\epsilon\to 0} \tfrac{1}{\epsilon}(\EXPECT_{P^{\theta+\epsilon v}}[f] - \EXPECT_{\PP}[f])\, .
\end{equation}

\begin{lemma}\label{c:star:lemma}
Assume Condition~\ref{con:smooth_densities} and let $v \in \R^k$. 

\noindent (i) Then
\begin{equation}
\RELENT{P^{\theta + v}}{\PP} = \frac{1}{2}\sum_{ij} \FISHER{\PP}_{ij} v_i v_j + \BIGO{(|v|^3)}\PERIOD
\end{equation}

\noindent (ii) Assume also that $f\in\ESET$ and thus the cumulant generating function $\CCGENF_{\PP,f}(c) \equiv \log\EXPECT_{\PP}[\EXP{c(f-\EXPECT_Pf)}]$
exists in a neighborhood of the origin, and that $f\not =\EXPECT_{P^\theta}[f]$ $P^\theta$-a.s. Then for $v \in \R^k$ and $\epsilon$ in a neighborhood of the origin
there exists a function $c^*(\epsilon)$  which is the unique solution of 
\begin{eqnarray*}%
&(P_+)\;\;\;\; &\inf_{c>0}\left\{\frac{1}{c}\CCGENF_{\PP,f}(c) + \frac{1}{c}\RELENT{P^{\theta+\epsilon v}}{\PP}\right\}\COMMA\\
\mbox{as well as}&& \\
&(P_-)\;\;\;\;&\sup_{c>0}\left\{-\frac{1}{c}\CCGENF_{\PP,f}(-c) - \frac{1}{c}\RELENT{P^{\theta+\epsilon v}}{\PP}\right\}\PERIOD
\end{eqnarray*}
Furthermore the function $c^*(\epsilon)$ admits the perturbation expansion 
\begin{equation}\label{pert:c:star1}
c^*(\epsilon) = c_1^* \epsilon + \BIGO(\epsilon^2)\COMMA
\end{equation}
where
\begin{equation}\label{optimalc}
c_1^* = \sqrt{\frac{\sum_{ij} \FISHER{\PP}_{ij}v_i v_j}{\VAR_\PP(f)}}\PERIOD
\end{equation}

\end{lemma}

\begin{proof}
The claim in (i) follows from (\ref{eqn:RE_expansion}) and (\ref{def:fisher}).
The claim in (ii) follows directly from Lemma~\ref{optim:asymptotics} after expanding the relative entropy
in $\epsilon$, i.e., writing $\rho^2(\epsilon) = \epsilon^2\frac{1}{2}\sum_{ij} \FISHER{\PP}_{ij}v_i v_j + \BIGO{(\epsilon^3)}$.
Substituting in \VIZ{c:star:rho} and \VIZ{optimalc_rho} we obtain \VIZ{pert:c:star1} and \VIZ{optimalc}.
\end{proof}

As a direct consequence of Theorem~\ref{linearization} we obtain a bound on the sensitivity indices by substituting $c^*(\epsilon)$ from \VIZ{pert:c:star1}
into $\Theta_\pm(c,\rho)$ (see the proof of Lemma~\ref{optim:asymptotics}).
\begin{theorem}\label{sensitivity:bound}
Under the assumptions of Lemma~\ref{c:star:lemma}, it holds that for $v \in \R^k$ and $\epsilon \not = 0$
\begin{equation}\label{lower:upper:bound:new}
\frac{1}{|\epsilon|}|\EXPECT_{P^{\theta+\epsilon v}}[f] - \EXPECT_{\PP}[f]| \leq \sqrt{\VAR_\PP(f)} \sqrt{\sum_{ij}\FISHER{\PP}_{ij}v_i v_j} + \BIGO(\epsilon)\COMMA
\end{equation}
and
\begin{equation}\label{sens:bound:3}
 |S_{f,v}(P^\theta)| \leq \sqrt{\VAR_\PP(f)} \sqrt{\sum_{ij}\FISHER{\PP}_{ij}v_i v_j}\PERIOD
\end{equation}
\end{theorem}
We refer to the inequality \VIZ{sens:bound:3} as a sensitivity bound of the sensitivity index $S_{f,v}(P^\theta)$.

\begin{remark}
{\rm
The bound \VIZ{lower:upper:bound:new} on the senitivity index is a direct consequence of 
more general non-infinitesimal bounds such as Theorem~\ref{linearization}. We note that 
in the special case of sensitivity analysis, where we consider small perturbations in the parameter space, we can obtain
sensitivity bounds of the same form as \VIZ{sens:bound:3} directly from the Cauchy-Schwarz inequality:
\begin{equation}\label{sens:bound:4a}
\begin{aligned}
|S_{f,v}(P^\theta)| &= \left|\frac{d}{d\epsilon} \EXPECT_P^{\theta+\epsilon v}[f]\right|
 = \left|\int (f(\omega) - \EXPECT_{\PP}[f]) \left(\frac{d}{d\epsilon} \log p^{\theta+\epsilon v}(\omega)\right) \PP(d\omega)\right| \\
&\leq \sqrt{\int (f(\omega) - \EXPECT_{\PP}[f])^2 \PP(d\omega )} \sqrt{\int\left( \frac{d}{d\epsilon} \log p^{\theta+\epsilon v}(\omega)\right)^2 \PP(d\omega )} \\
&= \sqrt{\VAR_{\PP}(f)}  \sqrt{\sum_{ij}\FISHER{\PP}_{ij}v_i v_j}\PERIOD
\end{aligned}
\end{equation}
Finally, we can also use \VIZ{chisquare_bound} applied to $\PP$ and $P^{\theta+\epsilon v}$ and obtain the same bound 
as in \VIZ{sens:bound:3}.
}
\end{remark}
\section{Path-space UQ information inequalities and sensitivity bounds}\label{pathspace}
In this section we develop new uncertainty quantification information inequalities and related sensitivity 
bounds
for stochastic processes and their path-dependent observables. The approach developed in the previous section
is applicable to obtaining similar bounds for functionals of Markov processes, when combined with path-space
Information Theory tools such as the RER and the associated path FIM.
These concepts which are discussed next were introduced as UQ and sensitivity analysis tools for stochastic
processes in \cite{Pantazis:Kats:13,PKV:2013,Kats:Plechac:13}.

\subsection{Information theory metrics in path space} 
We consider stochastic processes which are Markov and  
 take values in Polish space $\Xx$, although a much more general set up is also possible, see for instance \cite{Limnios:01}.
For simplicity in the presentation, we further restrict our discussion to discrete-time  Markov processes $\{X_t\}_{t\in\N_0}$ where  $\N_0=\N \cup \{0\}$
with the transition kernel $p(x,dy)$ and with the 
initial measure $\mu(dx)$, and the Markov process $\{Y_t\}_{t\in\N_0}$ 
with the transition kernel $q(x,dy)$  and with the stationary measure $\nu(dx)$. For the time interval $0,1,...,T$,
we denote by $\PPT$, $\QQT$ the respective probability measures on path space. 
Similar notation and constructions for all concepts introduced here will also be used when $t\in[0,\infty)$,
we refer to the Appendix A, as well as to \cite{Limnios:01,Pantazis:Kats:13}.

We will assume conditions under which the path-space relative entropy
$$\RELENT{\QQT}{\PPT}$$ 
is finite for all $T<\infty$. 
For stationary Markov processes, the relative entropy scales linearly
in $T$ as $T\to\infty$, \cite{Limnios:01}. Thus it is natural to define the concept of the rate of the relative entropy between path distributions.

\begin{definition}
Let  $\PPT$ and $\QQT$ be path-measures corresponding to Markov processes $\{X_t\}_{t\in \N_0}$, $\{Y_t\}_{t\in \N_0}$.
We define the relative entropy rate by
\begin{equation}\label{rer:def}
 \ENTRATE{Q}{P} = \lim_{T\to\infty} \frac{1}{T} \RELENT{\QQT}{\PPT}\COMMA
\end{equation}
when the limit exists.
\end{definition}

Although RER is a quantity between path distributions, we drop the dependence
of time interval in the notation of the RER because RER is a time-independent quantity.
Moreover, the relative entropy rate can often be expressed explicitly, which we demonstrate via
examples in Appendix A. For instance, in the case of discrete-time Markov Chains we have
\begin{equation}\label{entrate:MC:0}
\ENTRATE{Q}{P} %
= \int_{\mathcal X}\nu(dx)\int_{\mathcal X} q(x,dy)\log\frac{dq(x,\cdot)}{dp(x,\cdot)}(y)= \int_{\mathcal X}\RELENT{q(x,\cdot)}{p(x,\cdot)}\nu(dx)\PERIOD
\end{equation}

The significance of the definition of RER is elucidated by the following property of the relative entropy of two
path-measures for stationary processes. We state it for simplicity in the case of discrete-time Markov Chains, in which case it follows from the chain rule for relative entropy.
For the proof we refer to Appendix A. The proof was first given by Shannon in \cite{Shannon:48} and since then
has been extended in various directions for Markov and semi-Markov processes, \cite{Limnios:01}.
\begin{lemma}\label{RER:property}
Let $\{X_t\}_{t\in \N_0}$, $\{Y_t\}_{t\in \N_0}$ be two stationary Markov chains with the path-measures
$\PPT$ and $\QQT$. Suppose that $\nu$ is a stationary distribution for $\{Y_t\}$ and that the initial distribution $\mu$ of $\{X_t\}$ is arbitrary.
Then for any $T \in \N_0$
\begin{equation}\label{entrate:relent}
 \RELENT{\QQT}{\PPT} = T \ENTRATE{Q}{P} + \RELENT{\nu}{\mu}\COMMA
\end{equation}
and the relative entropy rate $\ENTRATE{Q}{P}$ is independent of $T$ and given by \VIZ{entrate:MC:0}.
\end{lemma}

As in Section~\ref{Sensitivity:2}, we will consider the sensitivity analysis problem, but this time in the context of both transient and stationary dynamics.
This amounts to an asymptotic expansion  of the relative entropy, and eventually the RER, in terms of a parameter perturbation.
First we consider the path-space probability measure $P^{\theta}_{[0,T]}$ where $\theta \in\R^k$ is a vector of the model parameters. 
We consider a perturbation $v\in\R^k$ in the parameter vector $\theta$ and the resulting path-space probability measure
$P^{\theta+v}_{[0,T]}$. We start out with two  regularity  conditions on the dependence of  probability measures on the parameter $\theta$; these 
conditions are not  the weakest possible, but they are fairly  simple to state.

\begin{condition} \label{con:smooth_densities_paths}
	There is a fixed reference probability measure $R \in \Pp(\Xx)$ such that $\P^\theta (x,dy) \ll R(dy)$ for all $x \in \Xx$ and $\theta \in \mathbb{R}^k$. Let $p^\theta (x,y) = \frac{dP^{\theta}(x,\cdot)}{dR(\cdot)}(y)$. Then we assume $(x,y,\theta) \rightarrow p^\theta (x,y)$ is continuous and for each fixed $x,y$ that $\theta \rightarrow p^\theta (x,y)$ is $C^{3}$.
Where needed, we also assume the existence of suitable dominating functions for various functions of $p^{\theta}$.
\end{condition}

Note that under this assumption, any stationary distribution $\mu^\theta$ will be absolutely continuous
with respect to $R$. 
It also holds that $\PPTE$ is absolutely continuous with respect to the product measure
on $\Xx^T$ with marginals $R$, with a smooth ($C^3$) Radon-Nikodym derivative. 
Condition~\ref{con:smooth_densities_paths} is necessary for the sensitivity results in finite and long times and it is directly verifiable, since it depends only on the local dynamics $p^\theta (x,y)$. However, for some of the 
results presented here for  infinite times, we additionally need Condition~\ref{con:smooth_densities_equil} below, which  
is a regularity condition for the stationary measure $\mu^{\theta}$ of the process $P^\theta_{[0, T]}$. Whenever this measure is analytically available, e.g., as  a Gibbs measure,  this condition is checkable directly. However, typically the stationary measure is not known and in this case this condition is not always easy to verify. Finally, conditions that ensure the regularity of the stationary measure $\mu^{\theta}$ and which rely  primarily on the existence of a spectral gap were given in \cite{Hairer:10}.

\begin{condition}\label{con:smooth_densities_equil}
	There is a fixed reference probability measure $R\in\mathcal{P}(\Omega)$ and for each $\theta$ a unique stationary probability measure $\mu^{\theta}$ such
	that $\mu^{\theta} \ll R$ and $\theta \rightarrow\frac{d\mu^{\theta}}{dR}(x)$ is $C^{3}$ for each fixed $x$.
		Where needed, we also assume the existence of
	suitable dominating functions for various functions of $\mu^{\theta}$.%
\end{condition}

Following \cite{Pantazis:Kats:13},
we define the path FIM for stationary Markov processes as the Hessian of the RER, at least when it exists:
\begin{equation}\label{hessian}
\FISHERR{P^\theta}:=\nabla_{v}^2\Big|_{v=0}\ENTRATE{P^\theta}{P^{\theta+v}}
= \nabla_{v}^2\Big|_{v=0}\ENTRATE{P^{\theta+v}}{P^\theta} \PERIOD
\end{equation}
In the case of a discrete-time Markov Chain,  under Condition~\ref{con:smooth_densities_paths} the path FIM reads (for a derivation
see Appendix A)
\begin{equation}
\label{RER:FIM}
\FISHERR{P^\theta} = \int\mu^\theta(dx) \int p^\theta(x,y)[\nabla_\theta \log p^\theta(x,y)][ \nabla_\theta \log p^\theta(x,y)]^TR(dy)\, .
\end{equation}
Notice that path FIM, just like   RER, e.g., \VIZ{entrate:MC:0}, can be  computed from the transition probabilities
under mild ergodic average assumptions, \cite{Pantazis:Kats:13}. Further examples of continuous-time Markov
processes are discussed in Appendix A.
Finally, using (\ref{entrate:relent}),  \VIZ{hessian} and Conditions~\ref{con:smooth_densities_paths} and \ref{con:smooth_densities_equil}, we have the expansion
\begin{equation}\label{relent:fim}
\frac{1}{T}\RELENT{\PPTE}{P^{\theta + v}_{[0,T]}} =  \ENTRATE{P^\theta}{P^{\theta+v}} +\frac{1}{T} \RELENT{\mu^\theta}{\mu^{\theta+v}} = \frac{1}{2} v^T \big(\FISHERR{P^\theta}
+\tfrac{1}{T}\FISHER{\mu^\theta}\big) v + \BIGO{(|v|^3)} \COMMA
\end{equation}
where $\FISHERR{P^\theta}$ is the path-space Fisher information \VIZ{RER:FIM}  while $\FISHER{\mu^\theta}$ is the
Fisher information for the stationary measure $\mu^\theta$, \VIZ{def:fisher}.

In the non-stationary regime we can use   the expansion (\ref{eqn:RE_expansion}) and Condition~\ref{con:smooth_densities_paths} to obtain for any initial measure $\rho$ of the stochastic process $P^\theta_{[0, T]}$ that is independent of $\theta$
\begin{equation}\label{relent:fim:transient}
\frac{1}{T}\RELENT{\PPTE}{P^{\theta + v}_{[0,T]}} = \frac{1}{2} v^T \FISHER{P^\theta_{[0, T]}}
 v + \BIGO{(|v|^3)} \, .
\end{equation}
Furthermore, assuming ergodicity of the process and similarly to  (\ref{rer:def}), we can obtain the path FIM %
as the asymptotic limit, \cite{Arampatzis:15},
\begin{equation}\label{fim:2}
\FISHERR{P^\theta} = \lim_{T\to\infty} \frac{1}{T} \FISHER{P^\theta_{[0, T]}}\, .
\end{equation}

\subsection{UQ Information inequalities for path-dependent observables}
We consider as an observable a measurable functional $\FOBS=\FOBS(\XXT)$ of the
process $\{X_t\}_{0\leq t \leq T}$. For any $T>0$ we define the centered observable
$$\tilde{\FOBS}(\XXT) =  \FOBS(\XXT) - \EXPECT_{\PPT}[\FOBS]\COMMA$$
and using the variational representation \VIZ{var_cumulant} of the cumulant-generating
function, we obtain for any $c>0$
\begin{equation}\label{cumulant:PS:1}
\CCGENF_{\PPT,T\FOBS}(c) \equiv \log \EXPECT_{\PPT} \EXP{{cT \tilde\FOBS}}
= \sup_{\QQT\ll\PPT} \big\{-\RELENT{\QQT}{\PPT}+ c T (\EXPECT_{\QQT}[\FOBS] - \EXPECT_{\PPT}[\FOBS])\big\} \PERIOD
\end{equation}
Concentrating on the stationary regime, the path-space relative entropy scales linearly
with time as shown  in  Lemma~\ref{RER:property}. Moreover, if we consider observables
for which  $\EXPECT_{\PPT}[\FOBS]$
and $\EXPECT_{\QQT}[\FOBS]$,
are  uniformly bounded for all $T$, then the second term
in the supremum in \VIZ{cumulant:PS:1} scales also linearly with time, therefore, the right hand side of the equation
scales at most linearly as $T\rightarrow\infty$ and its correct re-scaling for large times is given by
\begin{equation}\label{cumulant:PS:2}
\frac{1}{T}\CCGENF_{\PPT,T\FOBS}(c) = \sup_{\QQT\ll\PPT} \left\{-\frac{1}{T}\RELENT{\QQT}{\PPT}
+ c (\EXPECT_{\QQT}[\FOBS] - \EXPECT_{\PPT}[\FOBS])\right\} \PERIOD
\end{equation}
One class of such observables  which have a finite expectation as $T\rightarrow\infty$ is 
the case where $\FOBS$ is bounded by a constant.
Another class
of observables of this category which is of great interest in stochastic computing is that of ergodic averages:
\begin{equation}\label{average}
\FOBS(\XXT) = \frac{1}{T}\int_0^T f(X_{s})\,ds \, ,
\end{equation}
for a bounded observable   function $f$. Under suitable  ergodic assumptions we have 
\begin{equation}
\lim_{T\to \infty} \FOBS(\XXT) = \EXPECT_\mu[f] = \int f d\mu \PERIOD
\end{equation}

Next, we provide a result  on path space which is similar to \VIZ{lower:upper:bound:1}, first obtained
in \cite{Chowdhary:13,Li:12} for  measures $P$ and $Q$. We use the notation and the goal-oriented
divergence formulation in Theorem~\ref{divergence}. We show that for (suitable)  path-space observables, 
the analogue of relative entropy in \VIZ{lower:upper:bound:1} is now the concept of RER \VIZ{rer:def}.

\begin{theorem}\label{RER:bounds}
(a) (Finite-time regime) Assume  that the time-averaged cumulant generating function (\ref{cumulant:PS:2})%
exists in a neighborhood of the origin.
Define
\begin{eqnarray}
&   \XIQPpT &:= \inf_{c>0}\left\{\frac{1}{c T}\CCGENF_{\PPT,T\FOBS}(c) +\frac{1}{cT}\RELENT{\QQT}{\PPT}
\right\}\COMMA
\label{xi1}
\\
&  \XIQPmT &:= \sup_{c>0}\left\{-\frac{1}{c T}\CCGENF_{\PPT,T\FOBS}(-c) - \frac{1}{cT}\RELENT{\QQT}{\PPT}
\right\}
\label{xi2}
\PERIOD
\end{eqnarray}
Then we have the bounds
\begin{equation}\label{lower:upper:bound:2}
\XIQPmT \leq \EXPECT_{\QQT}[\FOBS] - \EXPECT_{\PPT}[\FOBS] \leq \XIQPpT \PERIOD
\end{equation}
In addition, based on 
Theorem~\ref{divergence} and
Theorem~\ref{representation} 
we have
\begin{equation}\label{repr:formula:PS}
\Xi_{\pm}(\QQT\SEP\PPT; \FOBS)=\CCGENF_{\PPT,T\FOBS}'\Big(\pm \Phi^{-1}
\big(\frac{1}{T}\RELENT{\QQT}{\PPT}\big)
 \Big)\, ,
\end{equation}
where $$\Phi(c):=-\CCGENF_{\PPT,T\FOBS}(c)+c\CCGENF_{\PPT, T\FOBS}'(c)$$ is a strictly increasing function
on $(0,\bar c)$, where $\bar c = \sup \{ c:\CCGENF_{\PPT,T\FOBS}(c)<\infty\}$.%

\noindent
(b) (Stationary regime) Consider the case of  stationary processes and assume the conditions of Lemma~\ref{RER:property}. Then in all formulas of part (a) we can substitute
\begin{equation}\label{rer:3}
\frac{1}{T} \RELENT{\QQT}{\PPT} =  \ENTRATE{Q}{P} + \frac{1}{T}\RELENT{\nu}{\mu}
\, .
\end{equation}
\end{theorem}
\begin{proof}
The proof follows immediately from Theorem~\ref{divergence} and Theorem~\ref{representation}, as well as the bounds 
\VIZ{lower:upper:bound:1} and the relative entropy rate representation of the relative entropy in Lemma~\ref{RER:property}, e.g., \VIZ{entrate:relent}.
\end{proof}

\subsection{Infinite time UQ bounds}
Here we discuss the extension of the previous UQ bounds to the stationary asymptotic  regime $T \to \infty$.
In the process, we demonstrate the key role in controlling the bounds played by the RER as well as connections
with the theory of Large Deviations, \cite{DZ, Dupuis:97}. First we state our primary assumptions. 

\begin{condition} \label{con:GE}
	For the centered cumulant-generating function  (\ref{cumulant:PS:1}), we assume the 
	$$\lim_{T\to\infty} \frac{1}{T} \CCGENF_{\PPT,T\FOBS}(c) = \CCGENF_{P,\FOBS}(c)$$
	 exists and is finite in a neighborhood of the origin $c=0$. 
	\end{condition}

It turns out that this  is also the main condition for the  G{\"a}rtner-Ellis Theorem in  Large Deviations, \cite{DZ}.
In this context, the limiting cumulant generating function $\CCGENF_{P,\FOBS}(c)$ can be calculated explicitly for various examples 
and through the Legendre transform it is associated with the large deviations  rate functional, \cite[Chapter 2.3]{DZ}. 
For example, in the case of a discrete-time, finite state space Markov chain given by the stochastic 
matrix $P=\big(p(x, y)\big)$ and the time-averaged observable $\FOBS=\frac{1}{T} \sum_{i=1}^Tf(X_i)$  
we have, (see  \cite[Chapter 3.1]{DZ}),
		\begin{equation}
		\CCGENF_{P,\FOBS}(c)=\log \lambda\big(\Pi_f(c)\big)\, .
		\end{equation}
Where $\lambda(B)$ denotes the Perron-Frobenius eigenvalue of the matrix $B$, and $\Pi_f(c)=\big(\pi_f(x, y; c)\big)$ 
the non-negative matrix with elements
$\pi_f(x, y; c)=p(x, y)\exp{(cf(y))}$. Due to the finiteness of the state space it is easy to show in this case
that $\CCGENF_{P,\FOBS}(c)$ is analytic and strictly convex in $c$, \cite{DZ}. 
		
Next we apply Condition~\ref{con:GE} and the asymptotics (\ref{rer:def}) to the transient regime bounds \VIZ{lower:upper:bound:2} 
in Theorem~\ref{RER:bounds} to obtain the following theorem for the $T\to\infty$ limit. 
The second statement follows along the lines of (\ref{repr:formula:PS}).

	\begin{theorem}\label{RER:bounds:GE}
		Assume  Condition~\ref{con:GE} and define
		\begin{eqnarray}
		&   \Xi_+(Q\SEP P; \FOBS) &:= \inf_{c>0}\left\{\frac{1}{c}\CCGENF_{P,\FOBS}(c) +\frac{1}{c} \ENTRATE{Q}{P}
		\right\}\COMMA
		\label{xi1:GE}
		\\
		&  \Xi_-(Q\SEP P; \FOBS) &:= \sup_{c>0}\left\{-\frac{1}{c }\CCGENF_{P, \FOBS}(-c) - \frac{1}{c} \ENTRATE{Q}{P}
		\right\}
		\label{xi2:GE}
		\PERIOD
		\end{eqnarray}
		Then,  we have the bounds
		\begin{equation}\label{lower:upper:bound:GE}
		\Xi_-(Q\SEP P; \FOBS) \leq \limsup_{T \to \infty} \Big(\EXPECT_{\QQT}[\FOBS] - \EXPECT_{\PPT}[\FOBS]\Big) \leq \Xi_+(Q\SEP P; \FOBS) \PERIOD
		\end{equation}
		In addition, similarly to 
		Theorem~\ref{representation} 
		we have
		\begin{equation}\label{repr:formula:GE}
		\Xi_{\pm}(Q\SEP P; \FOBS)=\CCGENF_{P,\FOBS}'\Big(\pm \Phi^{-1}
		\big(\ENTRATE{Q}{P}\big)\Big)
		\end{equation}
		where $$\Phi(c):=-\CCGENF_{P,\FOBS}(c)+c\CCGENF_{P, \FOBS}'(c)$$ is a strictly increasing function
		on $(0,\bar c)$, where $\bar c = \sup \{ c:\CCGENF_{P,\FOBS}(c)<\infty\}$.%

	\end{theorem}

\subsection{Linearization of the UQ bounds}
The bounds in Theorems~\ref{RER:bounds} and \ref{RER:bounds:GE} can become (asymptotically)
more explicit  in the case where the relative entropy or the RER $\ENTRATE{Q}{P}$ is small, that is
by expanding $\XIQPpmT$ in \VIZ{lower:upper:bound:2}. Furthermore, the RER can be explicitly calculated
in several examples discussed earlier in Section~\ref{pathspace} and in Appendix A. More specifically
we have the following asymptotics.

\begin{lemma}\label{optim:asymptotics:path:0}
Assume that  the cumulant generating function $\CCGENF_{\PPT,T\FOBS}(c)$ %
exists in a neighborhood of the origin. Assume also that  
$$
\frac{1}{T}\RELENT{\QQT}{\PPT}=\rho^2
$$ 
for two path probability measures  $\PPT$, $\QQT$. Note that by \VIZ{entrate:relent}, in the stationary case this is essentially an assumption on the relative entropy rate $\ENTRATE{Q}{P}$.
Then, there exists a function $c_T^*(\rho)$  which is the unique minimizer (resp. maximizer)  of \VIZ{xi1} (resp. \VIZ{xi2}).
Furthermore, there is $\rho_0>0$ such that $c_T^*(\rho)$ is $C^\infty$ in $(0,\rho_0)$ and admits the perturbation expansion 
\begin{equation}\label{optimalc_rho_2}
c_T^*(\rho) = c_{T,1}^*\rho + \BIGO(\rho^2)\, , \quad \mbox{where}\quad 
c_{T,1}^* =  \sqrt{\frac{2}{\frac{1}{T}\VAR_\PPT(T\FOBS)}}\PERIOD
\end{equation}
\end{lemma}

\begin{proof}
The proof %
follows the same steps as the proof of Lemma~\ref{optim:asymptotics}.
\end{proof}

Substituting the expansion in $\rho$ for the optimal value \VIZ{optimalc_rho_2} into the expansion of $\Xi_{\pm}(\QQT\SEP\PPT; \FOBS)$ we
obtain asymptotics of the upper and lower bounds for the weak error in $\rho^2 = \frac{1}{T}\RELENT{\QQT}{\PPT}$.
\begin{theorem}[Linearization]\label{linearization:PS}
Under the assumptions of Lemma~\ref{optim:asymptotics:path:0} we have:

\noindent (a) (Finite-time regime) 
\begin{equation}\label{weak_error_bound_path_trans}
|\EXPECT_{\QQT}[\FOBS] - \EXPECT_{\PPT}[\FOBS]|
\leq \sqrt{\frac{1}{T}\VAR_{\PPT}(T\FOBS)}\sqrt{\frac{2}{T}\RELENT{\QQT}{\PPT}}
+ \BIGO\Big(\frac{1}{T}\RELENT{\QQT}{\PPT}\Big)\COMMA
\end{equation}

\noindent (b) (Stationary regime) In this case  \VIZ{entrate:relent} implies 
\begin{equation}\label{weak_error_bound_path}
|\EXPECT_{\QQT}[\FOBS] - \EXPECT_{\PPT}[\FOBS]|
 \leq \sqrt{\frac{1}{T}\VAR_{\PPT}(T\FOBS)}\sqrt{2\left( \ENTRATE{Q}{P}+ \frac{1}{T}\RELENT{\nu}{\mu} \right)} + \BIGO\Big(\frac{1}{T}\RELENT{\QQT}{\PPT}\Big)\, .
\end{equation}
\end{theorem}
As in the static case, the term $\BIGO\Big(\frac{1}{T}\RELENT{\QQT}{\PPT}\Big)$ %
can be further quantified using the asymptotic expansions  of $c^*(\rho)$ and \VIZ{repr:formula:PS} in $\rho$.
\begin{remark}
{\rm For  observables which are time averages, e.g.,  $\FOBS(\XXT) = \frac{1}{T}\sum_{k=0}^{T-1}f(X_k)$, and 
for a stationary Markov process $P_{[0, T]}$ the variance terms in Theorem~\ref{linearization:PS} take the form of an autocorrelation function
and can be uniformly bounded in $T$. Specifically,
\begin{equation}
\label{IAT_T}
\frac{1}{T}\VAR_\PPT(T\FOBS)=\VAR_\mu(f)+2\sum_{k=1}^T \left(1-\frac{|k|}{T}\right) A_f(k):=\tau_T(f)\COMMA
\end{equation}
where $A_f(t) := \EXPECT_{\PPT} [(f(X_0)-\EXPECT_\mu[f(X_0)])(f(X_t)-\EXPECT_\mu[f(X_0)])] $
is the stationary covariance function. 	Recall that when $\tau(f):=\lim_{T \to \infty}\tau_T(f)<\infty$ 
then  $\tau(f)$  is known as the  the integrated autocorrelation function, \cite{Liu:MC}. 
The proof of \VIZ{IAT_T} is carried out in Lemma~\ref{IAT} below, see \VIZ{auto:T}.
}
\end{remark}

\subsection{Sensitivity bounds in path space}
\label{Sensitivity:3}
As in Section~\ref{Sensitivity:2}, we will consider the sensitivity analysis problem, but this time in both the context of transient and stationary dynamics.
First we consider the path-space probability measure $\PPTE$ where $\theta\in\R^k$ is a vector of the model parameters. 
We consider a perturbation $v\in\R^k$
in the parameter vector $\theta$ and the resulting path-space probability measure $P^{\theta +v}_{[0,T]}$, and focus first on the discrete time model.
The continuous time calculations are carried out in a similar manner, and we refer to Appendix A for the  related formulas.
The  next  theorem readily follows from  \VIZ{weak_error_bound_path_trans} and  \VIZ{weak_error_bound_path} and the asymptotics in $\epsilon$ in
(\ref{relent:fim}) and (\ref{relent:fim:transient}).
\begin{theorem}%
\label{finite-time}
Assume the  conditions of Lemma~\ref{optim:asymptotics:path:0}. 

\smallskip
\noindent(a) (Stationary regime) Furthermore, assume the conditions of Lemma \ref{RER:property} and Conditions~\ref{con:smooth_densities_paths} and \ref{con:smooth_densities_equil}. For any $v \in \R^k$ and $\epsilon \not = 0$, we have 
\begin{equation}\label{sens:bound:FD}
\frac{1}{|\epsilon |}
|\EXPECT_{P^{\theta+\epsilon v}_{[0,T]}}[\FOBS] - \EXPECT_{\PPTE}[\FOBS]|
\leq \sqrt{\frac{1}{T}\VAR_\PPTE(T\FOBS)} \sqrt{v^T \left(\FISHERR{P^\theta}+\frac{1}{T}\FISHER{\mu^\theta}\right) v} + \BIGO(\epsilon)\COMMA
\end{equation}
and %
\begin{equation}\label{sens:bound:pathwise:T}
 |S_{\FOBS,v}(\PPTE)| \leq 
 \sqrt{\frac{1}{T}\VAR_\PPTE(T\FOBS)} \sqrt{v^T \left(\FISHERR{P^\theta}+\frac{1}{T}\FISHER{\mu^\theta}\right) v}
 \COMMA
\end{equation}
where the sensitivity index $S_{\FOBS,v}(\PPTE)$ is defined in \VIZ{SI}.

\smallskip
\noindent(b) (Finite-time regime) If the process $P^\theta_{[0, T]}$  is not stationary then we only need to assume Condition~\ref{con:smooth_densities_paths}. Then we have the same bounds as in (a), however the term  $\sqrt{v^T \big(\FISHERR{P^\theta}+\frac{1}{T}\FISHER{\mu^\theta}\big) v}$ is replaced by the term $\sqrt{
	v^T\FISHER{P^\theta_{[0, T]}} v /T}$; we also note  the uniform bound in the  time horizon $T$ of the latter term due to (\ref{fim:2}).
\end{theorem}

We remark that in the stationary regime and for time-averaged observables such as \VIZ{average}, it holds that
\begin{equation}
S_{\FOBS,v}(\PPTE) = S_{f,v}(\mu^\theta)
\end{equation}
where $\mu^\theta$ is the stationary distribution, due to the regularity assumed in Condition~\ref{con:smooth_densities_equil}. 

\begin{remark}{\rm 
Bounds such as  \VIZ{sens:bound:pathwise:T} relate any stochastic gradient-type sensitivity analysis methods
such as
likelihood ratio \cite{Glynn:90},  Girsanov  \cite{Plyasunov:07} and  
path-wise methods \cite{Khammash:12} %
that develop efficient estimators for the sensitivity indices
\VIZ{SI}, with information theory based  methods,  showing that the latter provide a sensitivity bound on  \VIZ{SI}. 
Similarly the bound \VIZ{sens:bound:FD} relates sensitivity methods relying on finite-differencing \cite{Rathinam:10,Anderson:12,AK:2014}
with information-theory sensitivity analysis methods, \cite{Pantazis:Kats:13}. We refer to the inequalities
\VIZ{sens:bound:FD} and \VIZ{sens:bound:pathwise:T} as sensitivity bounds. These bounds  can be computed
efficiently and can provide fast  screening of insensitive observables, as well as parameters or directions in the
parameter space. We refer to \cite{AKP:2014} for more details, implementations and examples. %
}
\end{remark}

Next we focus on the infinite-time asymptotic regime and the related sensitivity bounds. Taking the
limit $T\to\infty$ we obtain bounds for time-averaged observables. First, we recall a result on the asymptotics of such  observables, \cite{Sokal:96}, and provide a proof for completeness in our presentation.
\begin{lemma}\label{IAT}
Under the assumptions of  Theorem~\ref{finite-time} and for observables of the
form 
\begin{equation}\label{ergodic:average}
\FOBS_T(\XXT) = \frac{1}{T}\sum_{i=0}^{T-1}f(X_i)\, ,
\end{equation} the following conclusions hold.
If the process $P^\theta_{[0, T]}$ is stationary and the  series $\sum_{k=1}^\infty  A_f(k)$ defined below converges absolutely, then  the limit  %
\begin{equation}
\lim_{T\rightarrow\infty} \frac{1}{T} \var_\PPTE\big(T\mathcal{F}_T(X)\big)=\tau(f)
\end{equation}
exists,  where $\tau(f)$ is the integrated autocorrelation function (IAT),
\cite{Sokal:96, Liu:MC}, defined as
\begin{equation}\label{IAT_infty}
\tau(f):=\lim_{T \to \infty}\tau_T(f) = \VAR_\mu^{\theta}(f)+2\sum_{k=1}^\infty  A_f(k)\COMMA
\end{equation}
and    
$$A_f(k) := \EXPECT_{\PPTE} [(f(X_0)-\EXPECT_{\mu^\theta}[f(X_0)])(f(X_k)-\EXPECT_{\mu^\theta}[f(X_0)])]$$
is the stationary covariance function of the process $X$.
\end{lemma}

\begin{proof}
A direct computation of the time-averaged variance gives
\begin{equation}
\begin{aligned}
\frac{1}{T} \var_\PPTE\big(T\mathcal{F}_T(X)\big)
&= \frac{1}{T} \mathbb E_\PPTE \left[\left(\sum_{i=0}^{T-1} f(X_i) - \mathbb E_\PPTE\left[\sum_{i=0}^{T-1} f(X_i)\right]\right)^2 \right] \\
&= \frac{1}{T} \sum_{i=0}^{T-1} \sum_{j=0}^{T-1} \EXPECT_\PPTE \big[(f(X_i)-\EXPECT_\PPTE[f(X_i)]) (f(X_j)-\EXPECT_\PPTE[f(X_j)])\big] \\
&= \frac{1}{T} \sum_{i=0}^{T-1} \sum_{j=0}^{T-1} \cov_f(i,j)
\end{aligned}
\end{equation}
where $\cov_f(i,j)$ is the covariance between $f(X_i)$ and $f(X_j)$.
Due to the stationarity,   we have  that under $\PPTE$ each $X_i$ is distributed according to $\mu^\theta$, hence  $\cov_f(i,j) = \EXPECT_{P_{[0,T]}^\theta} [(f(X_i)-\EXPECT_{\mu^\theta}[f(X_0)]) (f(X_j)-\EXPECT_{\mu^\theta}[f(X_0)])]
= \cov_f(i-j,0) \equiv A_f(i-j)$. Therefore,
\begin{equation}\label{auto:T}
\frac{1}{T} \var_\PPTE\big(T\mathcal{F}_T(X)\big) = \frac{1}{T} \sum_{k=-T+1}^{T-1} (T-|k|) \cov_f(k,0)
= \sum_{k=-T}^T \left(1-\frac{|k|}{T}\right) A_f(k)
\end{equation}
Sending $T\rightarrow\infty$, we obtain from  dominated convergence that
\begin{equation}
\lim_{T\rightarrow\infty} \frac{1}{T} \var_\PPTE\big(T\mathcal{F}_T(X)\big) = \sum_{k=-\infty}^\infty A_f(k)=\tau(f) \, .
\end{equation}
\end{proof}

	For stationary processes in continuous time, the formula for the IAT is given by
\begin{equation}
\tau(f) = \int_{-\infty}^{\infty} A_f(t) dt
\label{IAT:CTMC}
\end{equation}
where, as in the discrete time case,
$A_f(t) = \EXPECT_{\PPTE}[(f({X}_t)-\EXPECT_{\EQUIL}[f({x})])(f({X}_0)-\EXPECT_{\EQUIL}[f({x})])]$
is the stationary covariance between $f({X}_t)$ and $f({X}_0)$.

Now the following theorem readily follows  from \VIZ{sens:bound:pathwise:T}.
\begin{theorem}[Infinite-time]\label{inf:time:theorem}
Under Conditions~\ref{con:smooth_densities_paths} and \ref{con:smooth_densities_equil}, and the assumptions of the previous  lemma, the following hold.
For any $v \in \R^k$ %
\begin{equation}\label{sens:bound:pathwise}
 |S_{f,v}(\mu^\theta)| \leq 
 \sqrt{\tau(f)} \sqrt{v^T \FISHERR{P^\theta} v}
 \COMMA
\end{equation}
where we used the fact that $\EXPECT_{\PPTE}[\FOBS]=\EXPECT_{\mu^{\theta}}[f]$ for any stationary process
and therefore $S_{\FOBS,v}(\PPTE)=S_{f,v}(\mu^\theta)$ for the sensitivity indices defined by \VIZ{SI}.

\end{theorem}

\subsection{Some practical implications for sensitivity bounds}
\label{PracticalSB}
		Given the results in Section~\ref{Sensitivity:3}, as well as the computational feasibility of RER and path FIM 
		demonstrated in \cite{Pantazis:Kats:13,Kats:Plechac:13},
		we briefly  investigate sensitivity bounds  for  more general functionals than the time averages 
	 \VIZ{average} and less stringent conditions than those in Theorem~\ref{inf:time:theorem}.
		First, based on %
		Theorem~\ref{finite-time}(b)
		it follows that we can  consider any path-space observables $\FOBS$ such that  
		\begin{equation}\label{bound:trans}
		\frac{1}{T} \var_\PPTE\big(T\FOBS\big)=T\var_\PPTE\big(\FOBS\big) \le C^2< \infty\,  \quad
		\mbox{uniformly in $T$} \COMMA
		\end{equation}
		for some constant $C$.
		Next, using (\ref{fim:2}) we obtain from Theorem~\ref{finite-time}(b) and (\ref{bound:trans}) the limiting sensitivity bound
		\begin{equation}\label{sens:bound:pathwise:2}
		\limsup_{T \to \infty} |S_{\FOBS,v}(\PPTE)| \le {C}\sqrt{v^T \FISHERR{P^\theta} v} \, .
		\end{equation}
		Note that in contrast to Theorem~\ref{inf:time:theorem}, here we only need to assume the easily verifiable 
		Condition~\ref{con:smooth_densities_paths}, which depends solely  on the regularity in  $\theta$ of the local dynamics.
		Although the existence of the sensitivity index at $T=\infty$ is not guaranteed due to the absence  of (the hard to verify) 
		Condition~\ref{con:smooth_densities_equil}, or related conditions in \cite{Hairer:10}, the sensitivity indices  $S_{\FOBS,v}(\PPTE)$ remain controlled uniformly in time due to 
		(\ref{sens:bound:pathwise:2}). The boundedness of the variance associated with the observable $\FOBS$ in (\ref{bound:trans})
		can be monitored in the course of an actual simulation, while the path FIM in (\ref{sens:bound:pathwise:2}) is an easy to sample observable,
		as demonstrated in \cite{Pantazis:Kats:13}. Furthermore, the path FIM can for certain classes of stochastic dynamics 
		scale linearly with the number of model parameters, 
		making it computationally tractable even for systems with a very large number of parameters. For instance,   see   \cite{PKV:2013} for the case of complex biochemical reaction networks where the graph structure and the type of reaction rates  induce a block diagonal structure on the path FIM;  we also refer to Figure 1 in \cite{PKV:2013} for a demonstration.

		In a second direction geared also  towards practical implementation, we compare \VIZ{sens:bound:pathwise} and \VIZ{sens:bound:pathwise:2}  to the earlier static bound  \VIZ{sens:bound:3}.  Indeed, even though the form of the sensitivity bounds \VIZ{sens:bound:pathwise} and \VIZ{sens:bound:pathwise:2} are similar to \VIZ{sens:bound:3}, there are some substantial differences and advantages in considering the path-space bounds of this section.
		More specifically, when we want to study the sensitivity of ergodic   averages such as \VIZ{ergodic:average},
		we can either use the path space estimate  in \VIZ{sens:bound:pathwise}
		or alternatively
		the equilibrium bound \VIZ{sens:bound:3}, i.e.,
		\begin{equation}\label{sens:bound:3:2}
		|S_{f,v}(\mu^\theta)| \leq \sqrt{\VAR_{\mu^{\theta}}(f)}\sqrt{v^T \FISHER{\mu^{\theta}}v}\PERIOD
		\end{equation}
		On one hand, \VIZ{sens:bound:3:2} involves the FIM of the   equilibrium measures $ \mu^{\theta}$, which we do not typically have available in most non-equilibrium systems such as biochemical networks, reaction-diffusion mechanisms or driven systems. However, the path-wise estimate \VIZ{sens:bound:pathwise} can in principle always be computed since it involves only the local dynamics $p^\theta$ in the path FIM  \VIZ{RER:FIM}, see for instance \cite{Pantazis:Kats:13} and  \cite{PKV:2013}. 

\subsection{Cramer-Rao inequalities for time-series}
\label{Cramer:Rao:sec}
The sensitivity bounds \VIZ{sens:bound:pathwise:T} and \VIZ{sens:bound:pathwise} can be considered as extension of the
Cramer-Rao inequality for the time-series of Markov processes. Indeed, we recall that  for a parametric family
of probability measures $P^\theta$, where for simplicity in the presentation we assume that $\theta$ is scalar,
the Cramer-Rao inequality provides a lower bound for the variance of any unbounded statistical estimator.
Specifically, assume a biased estimator $\hat{\theta}=f(X)$ of the parameter $\theta$ with bias function $\psi(\theta)$,
i.e.,  $\EXPECT_{P^\theta} [f] =\psi(\theta)$. Then the Cramer-Rao bound for a scalar parameter $\theta$
states that, \cite{Kay:93},
\begin{equation}\label{CR:equil}
\VAR_{P^\theta}(\hat{\theta}) \ge \frac{[\psi'(\theta)]^2}{\FISHER{P^{\theta}}}\, .
\end{equation}
Upon rearranging, this bound is precisely the sensitivity bound \VIZ{sens:bound:3}, where the expected
value of the observable $f$ is the biased estimator of an unknown deterministic  parameter in a family
of probability measures. Furthermore, it is also known, \cite{Kay:93,Casella:2002},
that such bounds are sharp in the sense that for specific estimators (observables) such as the Maximum
Likelihood Estimator, the bound \VIZ{CR:equil}, \VIZ{sens:bound:3} becomes an equality.

In the same sense, we can obtain  a new Cramer-Rao type inequality  for time series stationary statistics
based on our UQ information bounds in path-space. Indeed, path-space observables such as
$\FOBS_T(X) = \frac{1}{T}\sum_i f(X_i)$ play the role of the statistical estimator for $\theta$, (i.e., $\hat{\theta}=\FOBS_T(X)$) and
the sensitivity bound \VIZ{sens:bound:pathwise} constitutes a Cramer-Rao lower bound for the IAT \VIZ{IAT_infty},
\begin{equation}\label{CR:PS}
\tau_{P^\theta}(f) \ge \frac{[\psi'(\theta)]^2}{\FISHERR{P^\theta}} %
\COMMA
\end{equation}
where $\psi(\theta)=\EXPECT_{\PPTE}[\FOBS_T]$ is the bias of the estimator. Therefore, for dependent
samples created for instance by Monte Carlo Markov Chain methods \cite{Liu:MC}, the lower bound \VIZ{CR:PS}
can be utilized. Finally, we remark that estimators with dependent samples have generally larger variance
than estimators using independent samples, however, for the same amount of computational time, larger
number of dependent samples than independent samples are drawn. Hence it is not clear which estimator
has better performance in terms of reduced variance for a given computational cost. In this direction, the
Cramer-Rao bound \VIZ{CR:PS} may be very useful.

\section{Demonstration Examples}
\label{demo:ex}
This section demonstrates the application of the derived bounds for several stochastic models.
The sensitivity bound derived in Section~\ref{Sensitivity:2} is utilized in the first two examples
where the sharpness of the bound is discussed. In the third and fourth examples, both stationary
and path space sensitivity bounds are computed and compared for various observable functions. In these examples,
the stationary bounds can be  slightly sharper than the bounds that utilize the path FIM, however,
stationary bounds are rarely explicitly available. Indeed, the birth/death process presented in
Section \ref{bd:sec} is a special case of a single-species biochemical reaction network with an 
explicit stationary distribution, however, for reaction networks with more species the stationary distribution
is generally unavailable.  Similarly, the stationary distribution in Section~\ref{ou:sec} where a stochastic differential
equation (SDE) example is considered,  is not generally known; for instance,  in SDE with additive noise where the drift term is not of conservative
type, i.e., the gradient of an appropriate function. For such stochastic models, the only 
available option for a tractable sensitivity bound is the path-space sensitivity bound \VIZ{sens:bound:pathwise}.

\subsection{Exponential family of distributions}
\label{exp:fam:sec}
A probability density function belongs to the exponential family if it admits the following canonical
decomposition \cite{Nielsen:09}
\begin{equation*}
\PP(x) = \exp\big\{t(x)^T\theta - F(\theta) + k(x)\big\}
\end{equation*}
where $t(x)=[t_1(x),...,t_K(x)]^T$ is the sufficient statistics vector, $\theta\in\mathbb R^K$ is the
parameter vector, $F(\cdot)$ is the log-normalizer (free energy in statistical
physics) and $k(x)$ is the carrier function (associated with the prior probability measure in statistical physics). The statistics $t(x)$ are called  ``sufficient" 
because it contains all the information needed for the estimation of the parameters.
Considering the sufficient statistics as observables, the corresponding sensitivity indices can be
analytically calculated as
\begin{equation*}
S_{t_k,\theta_l}(\PP) = \frac{\partial }{\partial\theta_l} \EXPECT_\PP[t_k(x)]
= \frac{\partial^2}{\partial\theta_k\partial\theta_l} F(\theta) \ \ ,\ \ \ k,l=1,...,K.
\end{equation*}
The covariance matrix of the sufficient statistics vector equals the Hessian of the log-normalizer,
$F$, (i.e., $\cov_\PP\big(t(x)\big) = \nabla^2  F(\theta)$), while the relative entropy of $\PP$ w.r.t.\
$\QQ$ can be written as the Bregman divergence of the log-normalizer on swapped natural
parameters \cite{Nielsen:09} given by
\begin{equation*}
\RELENT{\PP}{\QQ} = F(\theta+\epsilon)-F(\theta) - \epsilon^T \nabla F(\theta) \ .
\end{equation*}
A straightforward Taylor series expansion of $F$ in $\epsilon$ implies that the Fisher information
matrix, $\FISHER{\PP}$, defined in \VIZ{def:fisher} equals the Hessian of the log-normalizer, too.
Therefore, for sufficient statistics of the exponential family distribution,
Theorem~\ref{sensitivity:bound}  states that
\begin{equation}
\Big|S_{t_k,\theta_l}(\PP)\Big| = \Big|\frac{\partial^2}{\partial\theta_k\partial\theta_l} F(\theta)\Big| \leq
\sqrt{\var_\PP(t_k) \FISHER{\PP}_{l,l}} = \sqrt{\frac{\partial^2}{\partial\theta_k^2} F(\theta)\frac{\partial^2}{\partial\theta_l^2} F(\theta)} \ .
\label{exp:fam:bound}
\end{equation}
Notice that the inequality becomes an equality when $k=l$. From a parameter estimation perspective,
the equality of the bound of the $k$-th sufficient statistic with respect to the $k$-th
parameter is equivalent to the fact that $t_k(x)$ is an efficient estimator of $\theta_k$, $k=1,...,K$.
In other words, the Cramer-Rao bound (\ref{CR:equil})  is attained, \cite[Thm 5.12]{Lehmann:98}. Finally, another
bound for the sensitivity indices can be obtained directly from the properties of the Hessian of $F$:
the log-normalizer, $F$, is a strictly convex function \cite{Nielsen:09}, hence, its Hessian
is positive semi-definite which results in the bound,
$\big|S_{t_k,\theta_l}(\PP)\big| \leq
\frac{1}{2} \big(\frac{\partial^2}{\partial\theta_k^2} F(\theta) + \frac{\partial^2}{\partial\theta_l^2} F(\theta)\big)$.
However, this latter bound is less tight than the information-based bound \VIZ{exp:fam:bound} since the geometric mean is always less  or equal to
the arithmetic mean.

\subsection{Stochastic differential equation example}
\label{uq:sec}
We consider the differential equation
\begin{equation*}
\dot{u} = -Xu \ , \quad u(0)=u_0
\end{equation*}
where $X$ is a Gaussian random variable with mean $\mu$ and variance $\sigma^2$.
This stochastic model has been previously utilized for the assessment of uncertainty quantification
bounds in \cite{Li:12}. The stochastic solution of the equation is
\begin{equation*}
u(t) = u_0e^{-Xt}
\end{equation*}
whose distribution law is log-normal with parameters $\log(u_0)-\mu t$ and $(\sigma t)^2$.
The probability density function is given at time instant $t$ by
\begin{equation*}
\PP(u) = \frac{1}{u \sigma t\sqrt{2\pi}} \exp\left\{-\frac{(\log(u)-\log(u_0)+\mu t)^2}{2(\sigma t)^2}\right\}
\end{equation*}
where $\theta=[\mu,\sigma]^T$, and with the dependence of the density on time $t$ as well as on
the initial data $u_0$ is hidden for the sake of notational simplicity. The goal is to compute the observable
that quantifies the probability of $u(t)$ being larger that a determined value, $\bar{u}$, at time instant $t$.
This is a failure probability and can be written as an ensemble average,
\begin{equation}\label{failure}
P_f = \EXPECT_{\PP}[\chi_{\{u>\bar{u}\}}] \ ,
\end{equation}
where the observable function is the characteristic function (i.e., $f(u) = \chi_{\{u>\bar{u}\}}(u)$).
Notice that even though log-normal distribution belongs to the exponential family, here we are not interested 
 in the natural parameters or  the sufficient statistics, but we rather focus on the observable (\ref{failure}). Therefore, the
general setting of the previous subsection does not apply.
Nevertheless, calculations are still straightforward, and  the sensitivity index for $\mu$ is given by
\begin{equation*}
S_{\chi_{\{u>\bar{u}\}},\mu}(\PP) = -\EXPECT_{\PP}\Big[\chi_{\{u>\bar{u}\}}(u)\frac{\log(u)-\log(u_0)+\mu t}{\sigma^2t}\Big] \ ,
\end{equation*}
while the sensitivity index for the standard deviation $\sigma$ is
\begin{equation*}
S_{\chi_{\{u>\bar{u}\}},\sigma}(\PP) = \EXPECT_{\PP}\Big[\chi_{\{u>\bar{u}\}}(u)\frac{(\log(u)-\log(u_0)+\mu t)^2-\sigma^2t^2}{\sigma^3t^2}\Big] \ .
\end{equation*}
The variance of the observable is
\begin{equation*}
\var_{P^\theta}(\chi_{\{u>\bar{u}\}}) = \EXPECT_{\PP}[\chi^2_{\{u>\bar{u}\}}] - (\EXPECT_{\PP}[\chi_{\{u>\bar{u}\}}])^2
= \frac{1}{4} - \frac{1}{4}{\rm erf}^2\left(\frac{\log(\bar{u})-\log(u_0)+\mu t}{\sqrt{2}\sigma t}\right)
\end{equation*}
where ${\rm erf}$ is the error function while the diagonal elements of the FIM
for the log-normal distribution, $P^\theta$, are given by
\begin{equation*}
\FISHER{\PP}_{1,1} = \EXPECT_{\PP}\Big[\frac{(\log(u)-\log(u_0)+\mu t)^2}{(\sigma^2t)^2}\Big] \ ,
\end{equation*}
and
\begin{equation*}
\FISHER{\PP}_{2,2} = \EXPECT_{\PP}\Big[\frac{\big((\log(u)-\log(u_0)+\mu t)^2-\sigma^2t^2\big)^2}{(\sigma^3t^2)^2}\Big] \ .
\end{equation*}

Figures~\ref{ode:sigma:1} and \ref{ode:sigma:2} show the absolute value of the sensitivity indices
and the corresponding sensitivity bounds as a function of time for $\sigma=1$ and $\sigma=2$,
respectively. The remaining parameters were set to $u_0=1$, $\bar{u}=10$ and $\mu=1$ while the
computations of the expectations were performed numerically, whenever necessary. In Figure~\ref{ode:sigma:1},
the sensitivity bound of Theorem~\ref{sensitivity:bound} follows closely the sensitivity index in the
course of time. The sensitivity bound in Figure~\ref{ode:sigma:2} performs accurately for the sensitivity
index of the mean (upper panel), however, it is less sharp around time $t=5$ for the standard deviation
(lower panel) due to the existence of a zero transition of the sensitivity index at that particular instant.
Interestingly, the lower panel of Figure~\ref{ode:sigma:2} reveals that both upper and lower
bounds for small time and larger times, respectively, provide information about the corresponding sensitivity
index. Taking into account the complexity of the chosen observable which can be a risk-sensitive
(i.e., rare event) observable when $t$ is large, we would like to emphasize that even in this difficult case
there exists always a guaranteed bound for the sensitivity indices and in that sense one cannot but
benefit from its use.

\begin{figure}[!htb]
\begin{center}
\includegraphics[width=.8\textwidth]{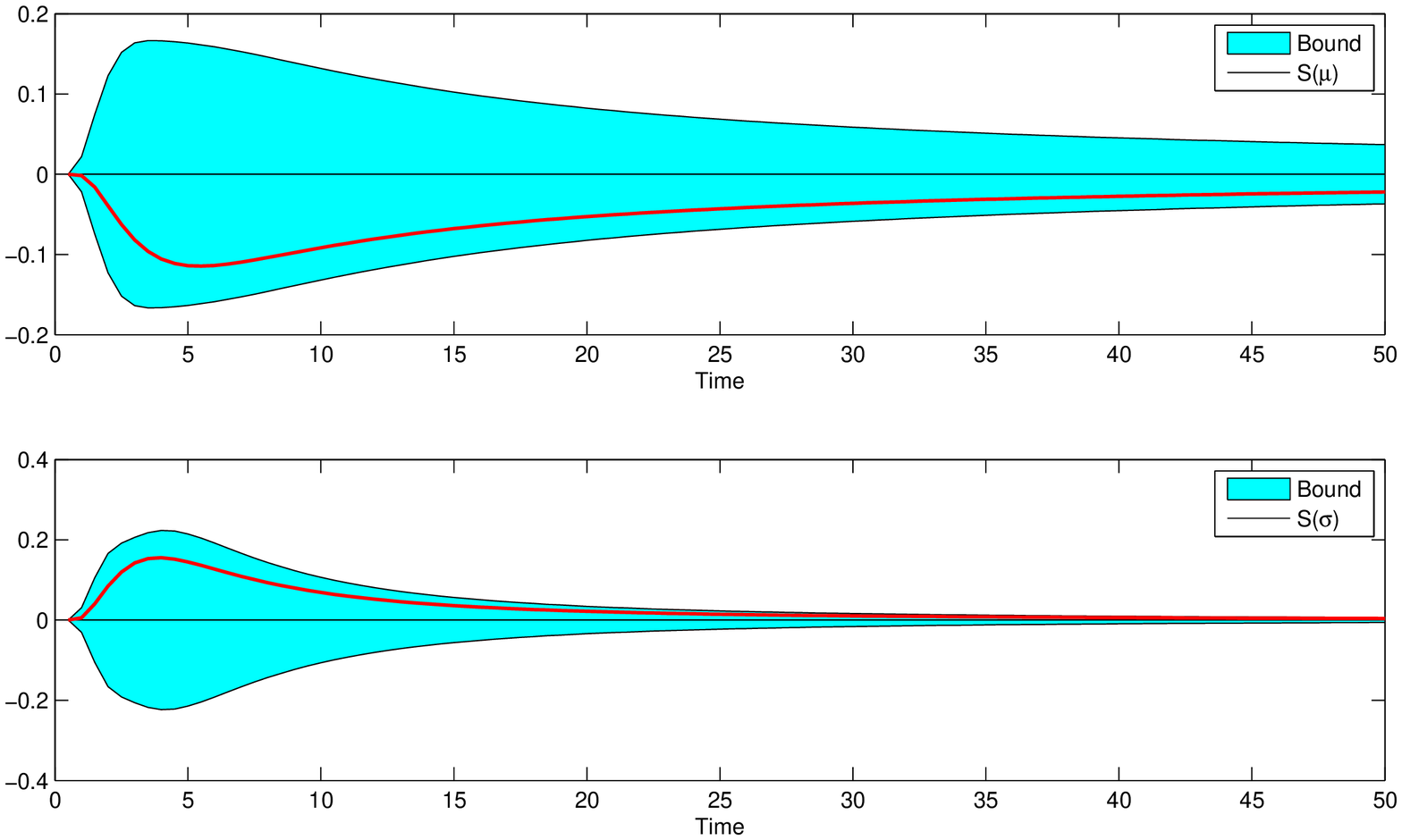}
\caption{Upper panel: Sensitivity index for the mean value (red) and the sensitivity bound
from Theorem~\ref{sensitivity:bound} (blue). Lower panel: Sensitivity index for the standard
deviation (red) and the respective sensitivity bound (blue). In both panels $\sigma=1$.}
\label{ode:sigma:1}
\end{center}
\end{figure}

\begin{figure}[!htb]
\begin{center}
\includegraphics[width=.8\textwidth]{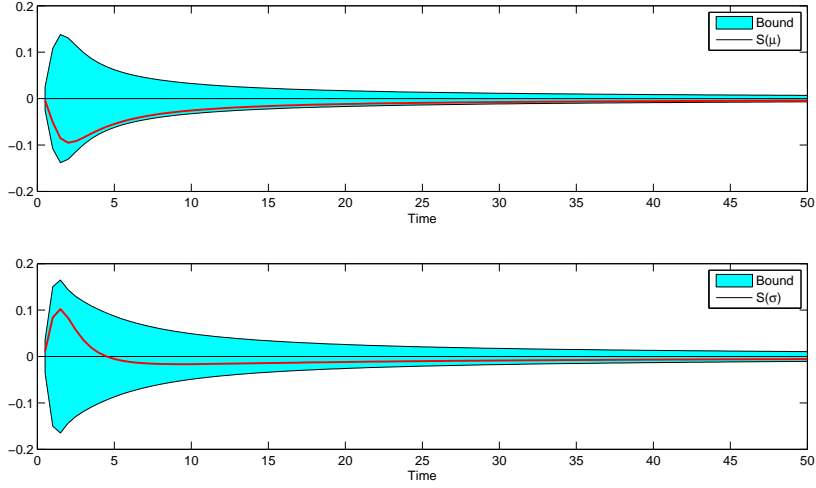}
\caption{Upper panel: Sensitivity index for the mean value (red) and the sensitivity bound
from Theorem~\ref{sensitivity:bound} (blue). Lower panel: Sensitivity index for the standard
deviation (red) and the respective sensitivity bound (blue). In both panels $\sigma=2$.}
\label{ode:sigma:2}
\end{center}
\end{figure}

\subsection{Birth/death process}
\label{bd:sec}
We consider a well-mixed reaction network which consists of one species and two reactions
given by
\begin{equation*}
\emptyset \underset{k_2}{\overset{k_1}{\rightleftarrows}} X \ .
\end{equation*}
The corresponding propensity functions for the  current state ${\bf x}=x$ are
\begin{equation*}
a_1(x) = k_1\quad \text{and} \quad a_2(x)=k_2x \ .
\label{Ex1:propensities}
\end{equation*}
Mathematically, this stochastic system is modeled as a continuous-time Markov chain (CTMC)
and due to its simplicity there exist analytic representations of the steady state (equilibrium)
distribution, moments and autocorrelation function, \cite[Sec. 7.1]{Gardiner:85}. The steady
state distribution, $\mu^\theta$, of the reaction network is Poisson with parameter $\frac{k_1}{k_2}$.
Hence, the steady state moments as well as the FIM for the parameter vector $\theta=[k_1,k_2]^T$
are known. The elements of the stationary FIM (eq.\ \VIZ{def:fisher}) are shown in
Table~\ref{FIMs:BD:table}. In the same Table, the elements of the path FIM
are shown \cite[pp.\ 10]{PKV:2013}.
Notice that the stationary FIM is singular while the path FIM is full rank implying that when
the complete time-series is provided then both parameters can be inferred. If samples
were i.i.d.\ drawn from the steady state distribution, then only the parameter ratio is inferable.%
\begin{table}[htdp]
\caption{Stationary and path-wise FIM's elements.}
\begin{center}
\begin{tabular}{|c||c|c|}\hline
Matrix element & Stationary FIM, $\FISHER{\mu^\theta}$ & Path FIM, $\FISHERR{P^\theta}$ \\ \hline\hline
$(1,1)$ & $\frac{1}{k_1k_2}$ & $\frac{1}{k_1}$ \\ \hline
$(1,2)$ & $-\frac{1}{k_2^2}$ & $0$ \\ \hline
$(2,2)$ & $\frac{k_1}{k_2^3}$ & $\frac{k_1}{k_2^2}$ \\ \hline
\end{tabular}
\end{center}
\label{FIMs:BD:table}
\end{table}
Next, we  consider two observables,  the mean, $f_1(x) = x$, and, the variance,
$f_2(x) = (x-\frac{k_1}{k_2})^2$. Since $\EXPECT_{\mu^\theta}[f_1]=\EXPECT_{\mu^\theta}[f_2] = \frac{k_1}{k_2}$,
the sensitivity indices are $S_{f_1,k_1}(\mu^\theta) = S_{f_2,k_1}(\mu^\theta) = \frac{1}{k_2}$ and
$S_{f_1,k_2}(\mu^\theta) = S_{f_2,k_2}(\mu^\theta) = -\frac{k_1}{k_2^2}$.
Moreover, in order to compute the IAT for $f_1$ and $f_2$, the computation of the
autocorrelation and the autocorrelation of the variance are necessary. Due to the linear
nature of this example, \cite{Gardiner:85}, explicit formulae exist and they are reported
in Table~\ref{VAR:IAT:BD:table}
The corresponding IATs are also shown in
Table~\ref{VAR:IAT:BD:table} for both observable functions.
\begin{table}[htdp]
\caption{Variance, autocorrelation function and IAT for the observables $f_1(x)$ and $f_2(x)$
of the birth/death process.}
\begin{center}
\begin{tabular}{|c|c|c|c|}\hline
Observable & Variance & ACF & IAT \\ \hline\hline
$f_1(x)=x$ & $\frac{k_1}{k_2}$ & $\frac{k_1}{k_2} e^{-k_2 |t|}$ & $2 \frac{k_1}{k_2^2}$ \\ \hline
$f_2(x)=(x-\frac{k_1}{k_2})^2$ & $\frac{k_1}{k_2}+2\frac{k_1^2}{k_2^2}$
& $\frac{k_1}{k_2} e^{-k_2 |t|}+2\frac{k_1^2}{k_2^2} e^{-2k_2 |t|}$ & $2(\frac{k_1}{k_2^2} + \frac{k_1^2}{k_2^3})$ \\ \hline
\end{tabular}
\end{center}
\label{VAR:IAT:BD:table}
\end{table}%
In Table~\ref{SB:BD:table}, both stationary and path-wise sensitivity bounds are compared to
the actual sensitivity indices. The  Poisson distribution belongs to the exponential family,
hence  we have a sharp bound for the mean value and the stationary case while the bound for the path-wise
case is worse by a $\sqrt{2}$ factor. When the variance is considered as observable, the stationary
bound is also slightly tighter than the path-wise bound. In the later case,
the path-wise bound becomes equivalent to the stationary bound when $k_2\ll k_1$, while both bounds
become sharper when $k_1\ll k_2$. Finally, note that even though we have  comparable performance for the 
stationary and path-wise bounds, there is a crucial advantage of the path-wise analysis which is its
computational tractability. Indeed, in complex reaction networks, the steady state distribution is
rarely known, hence, the stationary FIM cannot be derived. On the other hand, explicit formulas for
the path-wise FIM exists \cite{PKV:2013} and the corresponding sensitivity bound is computable
through Monte Carlo sampling.%
\begin{table}[htdp]
\caption{Sensitivity indices and the corresponding sensitivity bounds for the mean value and
the variance of the birth/death process.}
\begin{center}
\begin{tabular}{|c|c|c|c|}\hline
SI & Value & SB (Thm~\ref{sensitivity:bound}) & SB (Thm~\ref{inf:time:theorem}) \\ \hline\hline
$S_{x,k_1}(\mu^\theta)$ & $\frac{1}{k_2}$ & $\frac{1}{k_2}$ & $\sqrt{2}\frac{1}{k_2}$ \\ \hline
$S_{x,k_2}(\mu^\theta)$ & $-\frac{k_1}{k_2^2}$ & $\frac{k_1}{k_2^2}$ & $\sqrt{2} \frac{k_1}{k_2^2}$ \\ \hline
$S_{(x-\frac{k_1}{k_2})^2,k_1}(\mu^\theta)$ & $\frac{1}{k_2}$ & $\frac{1}{k_2}\sqrt{1+2\frac{k_1}{k_2}}$
& $\sqrt{2}\frac{1}{k_2} \sqrt{1+\frac{k_1}{k_2}}$ \\ \hline
$S_{(x-\frac{k_1}{k_2})^2,k_2}(\mu^\theta)$ & $-\frac{k_1}{k_2^2}$ & $\frac{k_1}{k_2^2} \sqrt{1+2\frac{k_1}{k_2}}$
& $\sqrt{2}\frac{k_1}{k_2^2} \sqrt{1+\frac{k_1}{k_2}}$ \\ \hline
\end{tabular}
\end{center}
\label{SB:BD:table}
\end{table}

\subsection{Ornstein-Uhlenbeck process}
\label{ou:sec}
Consider a one-dimensional Ornstein-Uhlenbeck (OU) process defined by the  stochastic
differential equation
\begin{equation*}
d X_t = -\alpha(X_t-\beta)dt + \gamma dB_t
\end{equation*}
where $\theta = [\alpha,\beta,\gamma]^T$ are the system's parameters
while $B_t$ is a one-dimensional Brownian motion. 
The stationary distribution of the OU process, $\mu^\theta$,
is Gaussian with mean $\beta$ and variance $\frac{\gamma^2}{2\alpha}$. The diagonal elements
of the stationary FIM %
are presented in Table~\ref{FIMs:OU:table}
(2nd column). Taking $f(x)=x$ as an observable, Table~\ref{VAR:IAT:OU:table} reports the variance
with respect to the stationary measure, the autocorrelation function as well as the IAT for the
continuous-time process. 

There are two  approaches for the computation of the path-wise FIM. The first is to
compute RER directly from the Girsanov formula and then the  FIM is obtained from a linearization
procedure. The formula for RER is given in \VIZ{RER:SDE}, thus, it is straightforward to calculate
the path-wise FIM whose diagonal elements are shown in Table~\ref{FIMs:OU:table}. Notice that
if the diffusion parameter, $\gamma$, is perturbed by a small amount then the RER is infinite. 
Indeed, by Girsanov's Theorem  the path-space measures of two SDE processes are not
absolutely continuous with each other when the diffusion terms are different, \cite{Karatzas:91,
Oksendal:00}. Therefore, the path-wise sensitivity bound in continuous-time is applicable only for
the parameters of the drift. Clearly, in the OU case a simple rescaling can remove the parameter from the noise term
and bypass altogether this issue.
The second approach is to discretize the stochastic process, defining a new discrete-time Markov chain
and then compute the path-wise FIM from the FIM of the DTMC renormalized with the time-step, \cite{Pantazis:Kats:13}.
Even though the second approach is an approximation, it is more flexible since it provides  a sensitivity bound even when the diffusion parameters
are considered. Overall, the time-discretization results in a regularization of the new
path-space measures, hence, a finite RER is obtained even if the parameters of the diffusion
part are perturbed. 
\begin{table}[htdp]
\caption{Variance, autocorrelation function and IAT for the mean value as an observable
of the OU process. Both continuous-time and discrete-time (Euler distretization) are considered.}
\begin{center}
\begin{tabular}{|c|c|c|c|c|c|}\hline
Observable & Variance & ACF (cont. time) & IAT (cont. time) & ACF (Euler) & IAT (Euler) \\ \hline\hline
$f(x)=x$ & $\frac{\gamma^2}{2\alpha}$ & $\frac{\gamma^2}{2\alpha} e^{-\alpha|t|}$ & $\frac{\gamma^2}{\alpha^2}$
& $\frac{\gamma^2}{2\alpha}\big(1-\alpha\Delta t\big)^n$ & $\frac{\gamma^2}{\alpha^2}$ \\ \hline
\end{tabular}
\end{center}
\label{VAR:IAT:OU:table}
\end{table}%
Following  the second approach, we consider the Euler
scheme for the OU process which is a first-order weak error integrator \cite{Kloeden:99} given at the $n$-th step by
\begin{equation*}
X_{n+1} = X_n + \alpha (X_n-\beta)\Delta t + \gamma\sqrt{\Delta t} \Delta W_n,
\end{equation*}
where $\Delta t$ is the discretization step while $\Delta W_n$ are i.i.d.\ zero-mean Gaussians
with unit variance. Hence, the transition probability, $p^\theta(x,y)$, is Gaussian with mean
$x+\alpha (x-\beta)\Delta t$ and variance $\gamma^2\Delta t$.
The last two columns of Table~\ref{VAR:IAT:OU:table} show the autocorrelation function as well
as the IAT for the discrete-time process obtained after discretization using the Euler scheme while
the last column of Table~\ref{FIMs:OU:table} shows the diagonal elements of the path-wise FIM
again for the same discrete-time process. In order to compute these quantities, averaging with respect to
the (unknown) stationary distribution of the Euler scheme, $\bar{\mu}^\theta$, which is an
approximation of the stationary distribution of the continuous-time process, $\mu^\theta$, is
required. However, we averaged with respect to $\mu^\theta$ instead of $\bar{\mu}^\theta$ exploiting
the fact that the produced weak error is of order $\BIGO(\Delta t)$, \cite{Mattingly:10}.
Another remark on the path-wise FIM is that when the limit $\Delta t\rightarrow 0$ is taken and the
diffusion parameter, $\gamma$, is perturbed then the corresponding FIM value is infinite which is
in accordance with the Girsanov Theorem restrictions mentioned earlier .
\begin{table}[htdp]
\caption{Diagonal elements of the stationary and path-wise FIMs for the Ornstein-Uhlenbeck process. Path-wise FIM
for both continuous-time and discrete-time approximation (Euler scheme) are considered.}
\begin{center}
\begin{tabular}{|c||c|c|c|}\hline
Matrix element & Stationary FIM & Path FIM (cont. time) & Path FIM (Euler) \\ \hline\hline
$(1,1)$ & $\frac{1}{2\alpha^2}$ & $\frac{1}{2\alpha}$ & $\frac{1}{2\alpha}$ \\ \hline
$(2,2)$ & $\frac{2\alpha}{\gamma^2}$ & $\frac{\alpha^2}{\gamma^2}$ & $\frac{\alpha^2}{\gamma^2}$ \\ \hline
$(3,3)$ & $\frac{2}{\gamma^2}$ & $\infty$ & $\frac{2}{\gamma^2\Delta t}$ \\ \hline
\end{tabular}
\end{center}
\label{FIMs:OU:table}
\end{table}%

Table~\ref{SB:OU:table} presents the sensitivity indices and the various sensitivity bounds for the
mean value as an observable.
The stationary bound for $\beta$ is sharp as expected due to the fact
that Gaussian belongs to the exponential family and the mean value is a sufficient statistic. The continuous-time
path-wise bound as well as the discrete-time path-wise bound (up to order $O(\Delta t)$) for $\beta$
are sharp. For $\alpha$, the stationary bound is smaller by a factor of $\sqrt{2}$ while for $\gamma$
the factor $\frac{\sqrt{2}}{\sqrt{\Delta t}}$ of the discrete-time path-wise bound
make the stationary bound better. Finally, notice that as in the birth/death process the stationary
bounds are slightly tighter. However, for general SDEs where the drift term is not necessarily of
conservative type, the stationary distribution is rarely known hence the computation of stationary
FIM and consequently the stationary bounds are intractable. For instance, a large class of stochastic
processes where the stationary distribution is not  known consists of the non-equilibrium systems where
the drift is a non-conservative force while the noise is additive, \cite{Qian:06}, \cite{Maes:00}. Therefore, the respective stationary sensitivity bound
is intractable for this important category of stochastic processes, while the path-wise bound (\ref{sens:bound:pathwise}) is computable.

\begin{table}[htdp]
\caption{Sensitivity indices and the corresponding sensitivity bounds for the mean value of
the OU process.}
\begin{center}
\begin{tabular}{|c|c|c|c|c|}\hline
SI & Value & SB (Thm~\ref{sensitivity:bound}) & SB (Thm~\ref{inf:time:theorem}, cont. time) & SB (Thm~\ref{inf:time:theorem}, Euler) \\ \hline\hline
$S_{x,\alpha}(\mu^\theta)$ & $0$ & $\frac{\gamma}{2\alpha\sqrt{\alpha}}$ & $\sqrt{2}\frac{\gamma}{2\alpha\sqrt{\alpha}}$
& $\sqrt{2}\frac{\gamma}{2\alpha\sqrt{\alpha}}$ \\ \hline
$S_{x,\beta}(\mu^\theta)$ & $1$ & $1$ & $1$ & 1 \\ \hline 
$S_{x,\gamma}(\mu^\theta)$ & $0$ & $\frac{1}{\sqrt{\alpha}}$ & $\infty$ & $\frac{\sqrt{2}}{\alpha\sqrt{\Delta t}}$ \\ \hline
\end{tabular}
\end{center}
\label{SB:OU:table}
\end{table}

\section{Conclusions}\label{conclusions}
In this paper, we derived information inequalities that bound weak error estimates and sensitivity indices.
We further extend the variational UQ bounds which were previously derived in \cite{Chowdhary:13, Li:12}
in several directions.
First, we observe and prove that the UQ bound defines a novel {\it goal-oriented} divergence which couples
observables of interest  (hence the term ``goal-oriented'') with the relative entropy of the ``true'' probabilistic model  with respect to a computationally tractable  ``nominal''
model. Second, an explicit representation for the goal-oriented divergence was derived 
which after linearization 
resulted in a sensitivity bound which decouples the role of the observable function from 
the distance of the probability measures as quantified by the FIM. 
Exploiting the properties of the relative entropy in path-space, we further extend the UQ and sensitivity bounds
to the case of stochastic dynamics for both transient and long-time regimes. The relative entropy rate which is the relative entropy
per unit time and the corresponding path FIM are the quantities that control the weak error and the sensitivity
indices, respectively, at infinite times. An advantage of the path-space sensitivity bounds is that they depend only on the local
dynamics of the process thus they are computable from
a direct Monte Carlo simulation. This feature is very attractive  in  out-of-equilibrium or  non-equilibrium systems, where the stationary distribution is not relevant or known.
Finally,  this paper is primarily a theoretical work and extensive numerical examples,  algorithms, synergies with other methods 
and applications to high-dimensional realistic systems will follow.

\appendix

\section{Relative entropy rate and path Fisher information matrix: Examples}
\label{appendix:sec}
The relative entropy rate (RER) and the path Fisher Information Matrix (pFIM) can often be expressed explicitly
in terms of the local dynamics, which we demonstrate in a few examples for Markov processes, including discrete
and continuous time Markov chains and stochastic differential equations.

\subsection{Discrete-time Markov chains}
\label{DTMC:app}
RER always has an explicit expression for discrete time processes with values in the Polish space $\Xx$.
We first state a version of the chain rule. For a proof see \cite[Theorem C.3.1]{Dupuis:97}.

\begin{lemma}
\label{lem:ch_rule}
Let $\alpha$ and $\beta$ be probability measures on $\Xx \times \mathcal{Y}$, where $\Xx$ and $\mathcal{Y}$
are Polish spaces. Let $\alpha_1$ and $\beta_1$ denote their first marginals, and denote by $\alpha(dy|x)$ and
$\beta(dy|x)$ the conditional distribution on the second variable given the first. Then the mapping
$x\rightarrow \RELENT{\alpha(\cdot|x)}{\beta(\cdot|x)}$ is measurable, and 
	\[
	\RELENT{\alpha}{\beta}=\RELENT{\alpha_1}{\beta_1}+\int_{\Xx} \RELENT{\alpha(\cdot|x)}{\beta(\cdot|x)}\alpha_1(dx).
	\]
\end{lemma}

\begin{lemma}
	Let  $\{X_t\}_{t\in\N_0}$, $\{Y_t\}_{t\in\N_0}$ be Markov processes on the state space $\mathcal{X}$ with transition kernels
	$q(x,dx')$ and $p(x,dx')$, and initial measures  $\nu(dx)$ and  $\mu(dx)$, respectively. Assume that $\nu$ is stationary for $q(x,dx')$.
	Then the relative entropy rate $\ENTRATE{Q}{P}$ defined in \VIZ{rer:def} is given by
	\begin{equation}\label{entrate:MC}
		\ENTRATE{Q}{P} %
		= \int\nu(dx)\int q(x,dy)\log\frac{dq(x,\cdot)}{dp(x,\cdot)}(y) \PERIOD
	\end{equation}
	Furthermore, the relative entropy rate is expressed as the relative entropy 
	\begin{equation}\label{entrate:relent2}
		\ENTRATE{Q}{P} = \RELENT{\nu\otimes q}{\mu\otimes p}\COMMA
	\end{equation}
	where $\nu\otimes q$ is the probability measure on $\Xx^2$ given by $[\nu\otimes q](A \times B)=\int_A q(x,B)\nu(dx)$.
\end{lemma}

\begin{proof}
	Both statements follow directly from the chain rule, Lemma \ref{lem:ch_rule}. Since $\nu$ is stationary for $q(x,dx')$, we can apply the chain rule
	from time $t=T-1$ back to $t=0$, and by using Markov property obtain  (\ref{entrate:relent}), with $\ENTRATE{Q}{P}$ equal to 
	\[
	\int_{\Xx} \RELENT{q(x,\cdot)}{(p(x,\cdot)}\nu(dx).
	\]
	However this is precisely (\ref{entrate:MC}), and thus the first claim follows.  (\ref{entrate:relent2}) also follows directly from the chain rule and the fact that $\RELENT{\nu}{\nu}=0$.
	Finally, notice that even though a quantity between path distributions, we drop the dependence of time interval
	in the notation of the relative entropy rate because relative entropy rate is a time-independent quantity.
\end{proof}

\begin{lemma}
Assume Condition~\ref{con:smooth_densities_paths}. Then, the path FIM defined in \VIZ{hessian} is given by
\begin{equation}
\FISHERR{P^\theta} =  \mathbb E_{\mu^\theta} \left[\int_E p^\theta(x,y) \nabla_\theta
\log p^\theta(x,y) \nabla_\theta \log p^\theta(x,y)^T \,R(dy)\right]
\end{equation}
\end{lemma}

\begin{proof}
Define the function $G(\theta)=G(\theta;x,y) = \log p^\theta(x,y)$ for all $x,y\in\Xx$.
Then, from Condition~\ref{con:smooth_densities_paths}, $G(\theta)$ as a function
of $\theta$ is $C^3$ and for an arbitrary $\epsilon\in\mathbb R^k$
\begin{equation*}
\begin{aligned}
G(\theta+\epsilon) &= G(\theta) + \epsilon^T \nabla_\theta G(\theta)
+ \frac{1}{2}\epsilon^T \nabla_\theta G(\theta) \epsilon + R_2(\theta) \\
&= G(\theta) + \epsilon^T \frac{\nabla_\theta p^\theta}{p^\theta}
+ \frac{1}{2}\epsilon^T \Big( \frac{\nabla_\theta^2 p^\theta}{p^\theta} - \Big(\frac{\nabla_\theta p^\theta}{p^\theta}\Big)^2 \Big) \epsilon + R_2(\theta) \ ,
\end{aligned}
\end{equation*}
where $\nabla$ and $\nabla^2$ denotes the gradient and the Hessian of a function
while $R_2(\theta)$ is the remainder term as given by Taylor's Theorem.
Then, the relative entropy rate of the path distribution $P^\theta_{[0,T]}$ with respect to
the perturbed path distribution $P^{\theta+\epsilon}_{[0,T]}$ becomes
\begin{equation*}
\begin{aligned}
&\ENTRATE{P^\theta}{P^{\theta+\epsilon}} = \int\mu^\theta(dx)\int p^\theta(x,y) \log\frac{p^\theta(x,y)}{p^{\theta+\epsilon}(x,y)} \, R(dy) \\
&= - \int\mu^\theta(dx)\int p^\theta(x,y) (G(\theta+\epsilon;x,y) - G(\theta;x,y)) \, R(dy)\\
&= - \int\mu^\theta(dx)\int p^\theta(x,y) \left( \epsilon^T \frac{\nabla_\theta p^\theta(x,y)}{p^\theta(x,y)}
+ \frac{1}{2}\epsilon^T \Big( \frac{\nabla_\theta^2 p^\theta(x,y)}{p^\theta(x,y)}
- \Big(\frac{\nabla_\theta p^\theta(x,y)}{p^\theta(x,y)}\Big)^2 \Big) \epsilon + R_2(\theta;x,y) \right) \, R(dy) \\
&= \frac{1}{2}\epsilon^T \int\mu^\theta(dx)\int p^\theta(x,y) \Big(\frac{\nabla_\theta p^\theta(x,y)}{p^\theta(x,y)}\Big)^2 R(dy) \epsilon
+ \int\mu^\theta(dx)\int p^\theta(x,y) R_2(\theta;x,y)\, R(dy)
\end{aligned}
\end{equation*}
since for any $i=1,2,...$ it holds that
\begin{equation*}
\int p^\theta(x,y) \frac{\nabla_\theta^i p^\theta(x,y)}{p^\theta(x,y)} R(dy)
= \int \nabla_\theta^i p^\theta(x,y) R(dy)
= \nabla_\theta^i \int p^\theta(x,y) R(dy)
= \nabla_\theta^i 1
= 0
\end{equation*}
where $\nabla_\theta^i$ denotes the $i$-th derivative operator.
Thus, the path FIM is given by
\begin{equation*}
\FISHERR{P^\theta} =  \mathbb E_{\mu^\theta} \left[\int p^\theta(x,y) \nabla_\theta
\log p^\theta(x,y) \nabla_\theta \log p^\theta(x,y)^T \,R(dy)\right]
\end{equation*}
\end{proof}

\begin{remark}
Performing similar Taylor series expansion, it can be obtained that the relative entropy rate of $P^{\theta+\epsilon}$
w.r.t. $P^\theta$ admits the same Hessian. Indeed, it is expanded as
\begin{equation*}
\ENTRATE{P^{\theta+\epsilon}}{P^\theta} = \frac{1}{2}\epsilon^T \FISHERR{P^\theta} \epsilon + \BIGO(|\epsilon|^3)\PERIOD
\end{equation*}
Notice also that this result is valid not only for discrete-time Markov chains but it is quite general.
\end{remark}

\subsection{Continuous-time Markov chains}
\label{CTMC:app}
Next, we compute the relative entropy rate for continuous-time
Markov chains. We consider such chains on a countable state space
$\mathcal{X}$ and let quantities such as $P_{[0,T]}$ denote the measure on
$D([0,T]:\mathcal{X})$ induced by the process, where $D([0,T]:\mathcal{X})$
consists of all $X:[0,T]\rightarrow\mathcal{X}$ that are continuous from the
right and with limits from the left, with the usual Skorohod topology. 
\begin{lemma}
	Let $\{X_{t}\}_{t\geq0}$ and $\{Y_{t}\}_{t\geq0}$ be stationary continuous
	time Markov chains with the countable state space $\mathcal{X}$ and jump rates
	$\tilde{\lambda}(x)$ and $\lambda(x)$ and transition probabilities
	$\tilde{p}(x,x^{\prime})$ and $p(x,x^{\prime})$. Assume that $\tilde{\lambda}$ and
	$\lambda$ are positive and uniformly bounded above. Assume also that
	$\tilde{p}(x,x)=p(x,x)=0$ for all $x\in\mathcal{X}$, and for $x'\neq x$ that
	$\tilde{p}(x,x')>0$ if any only if $p(x,x')>0$. Let $\tilde{\mu}$ be a stationary
	probability distribution for $\{X_{t}\}_{t\geq0}$, and let $\mu$ be
	any initial distribution for $\{Y_{t}\}_{t\geq0}$. Let $Q_{[0,T]}$ and
	$P_{[0,T]}$ be the measures induced by $\{X_{t}\}_{t\geq0}$ and $\{Y_{t}%
	\}_{t\geq0}$. Then the relative entropy rate $\mathcal{H}(Q
	\!\parallel\!P)$ associated with $\mathcal{R}%
	(Q_{[0,T]}\!\parallel\!P_{[0,T]})$ is given by
	\begin{equation}\label{entrate:CTMC}
		\mathcal{H}(Q\!\parallel\!P)=\sum_{x\in\mathcal{X}}\sum_{x^{\prime}\in\mathcal{X}}\tilde{\mu}(x)\tilde{\lambda}(x)\tilde{p}(x,x^{\prime})
		\log\frac{\tilde{\lambda}(x)\tilde{p}(x,x^{\prime})}{\lambda(x)p(x,x^{\prime})}-\sum_{x\in\mathcal{X}}\tilde{\mu}(x)
		(\tilde{\lambda}(x)-\lambda(x))\PERIOD
\end{equation}
\end{lemma}

\begin{proof}
	According to \cite[Prop. 2.6, App. 1]{Kipnis:99} and \cite[Sec. 19]{Liptser:77} the Radon-Nikodym
	derivative of the path measure $Q_{[0,T]}$ with respect to the path measure $P_{[0,T]}$ is given by
	$$
	\frac{d\QQT}{d\PPT}(X) = \frac{\tilde\mu(X_0)}{\mu(X_0)}\exp\left\{
		\int_0^T \log \frac{\tilde\lambda(X_t) \tilde p(X_{t_-},X_t)}{\lambda(X_t) p(X_{t_-},X_t)}\, dN_t(X)
		- \int_0^T (\tilde\lambda(X_t) - \lambda(X_t))\, dt\right\}\COMMA
	$$
	where $N_s(X)$ is the number of jumps on the path $X$ up to time $s$.
	The relative entropy up to time $T$ is defined by
	$$
	\RELENT{\QQT}{\PPT}\equiv \EXPECT_{\QQT}\left[\log \frac{d\QQT}{d\PPT}\right]\PERIOD
	$$
	Since $\tilde\lambda$ is bounded 
	$M_T \equiv N_T - \int_0^T \tilde\lambda(X_t)\, dt$ is a mean zero martingale, then 
	for any (non-negative and measurable) function $f$ on $\Xx$
	$$
	\EXPECT_{\QQT}\left[\int_0^T f(X_t)\, dN_t\right] = \EXPECT_{\QQT}\left[\int_0^T f(X_t) \tilde\lambda(X_t)\, dt\right] \PERIOD
	$$
	Furthermore, from stationarity %
	we have $\EXPECT_{\QQT}[\int_0^T f(X_t) \tilde\lambda(X_t)\, dt] = 
	T\sum_{x\in\Xx} \tilde\mu(x) f(x)\tilde\lambda(x)$.
	Substituting for $f$ the expression for the logarithm of the Radon-Nikodym derivative we obtain 
	$$
	\RELENT{\QQT}{\PPT} = T\left(\sum_{x\in\mathcal{X}}\sum_{x^{\prime}\in\mathcal{X}}\tilde{\mu}(x)\tilde{\lambda}(x)\tilde{p}(x,x^{\prime})
		\log\frac{\tilde{\lambda}(x)\tilde{p}(x,x^{\prime})}{\lambda(x)p(x,x^{\prime})}-\sum_{x\in\mathcal{X}}\tilde{\mu}(x)
		(\tilde{\lambda}(x)-\lambda(x))\right) + \RELENT{\tilde\mu}{\mu}\PERIOD
	$$
\end{proof}

\begin{remark}
	{\rm
		We can rearrange the expression for the RER to obtain 
		\[
		\mathcal{H}(Q\!\parallel\!P)=
		\sum_{x\in\mathcal{X}}\sum_{x^{\prime}\in\mathcal{X}}\tilde{\mu}(x)\lambda(x)
		p(x,x^{\prime})\ell\left(  \frac{\tilde{\lambda}(x)\tilde{p}(x,x^{\prime})}{\lambda(x)p(x,x^{\prime})}\right)  ,
		\]
		where
	$	\ell\left(  z\right)  =z\log z-z+1 \,\,\mbox{ for }z \geq 0$.
		This exhibits the RER as a form of relative entropy.  The function $\ell\left(  z\right)$, which appears in rate functions 
		for the large deviation theory of jump Markov processes \cite{Dupuis:97}, is non-negative and
		vanishes only at $z=1$. Thus the RER is non-negative, and equals zero if and only if the two chains are the same.
	}
\end{remark}

\begin{lemma}
Let the transition rate defined for all $x,x'\in\Xx$ by $c^\theta(x,x') \equiv \lambda^\theta(x)p^\theta(x,x')$ be parametrized
by $\theta\in\mathbb R$ and ssume that the mapping $\theta\rightarrow c^\theta(\cdot,\cdot)$ is $C^3$. Let $P^\theta_{[0,T]}$
(resp. $\mu^\theta$) be the path (resp. stationary) measure of the associated process. Then, the path FIM is
\begin{equation}
\FISHERR{P^\theta} = \mathbb E_{\mu^{\theta}}\left[ \sum_{x'\in \Xx} c^\theta(x,x')
\nabla_\theta \log c^\theta(x,x') \nabla_\theta \log c^\theta(x,x')^T \right]
\label{path:FIM:CTMC}
\end{equation}
\end{lemma}

\begin{proof}
The proof is similar to the DTMC case using now two auxiliary functions defined by $G_1(\theta) = G_1(\theta;x,x') =  \log c^\theta(x,x')$
and $G_2(\theta) = G_2(\theta;x,x') =  c^\theta(x,x')$ for all $x,x'\in\Xx$. For completeness, we present the basic steps
of the relative entropy expansion. The relative entropy rate of the path measure $P^\theta_{[0,T]}$ with respect to
the perturbed path measure $P^{\theta+\epsilon}_{[0,T]}$ can be written as
\begin{equation*}
\begin{aligned}
&\ENTRATE{P^\theta}{P^{\theta+\epsilon}} = \sum_{x,x'\in\mathcal{X}}\mu^\theta(x)c^\theta(x,x^{\prime})
\log\frac{c^\theta(x,x')}{c^{\theta+\epsilon}(x,x')} - \sum_{x,x'\in\mathcal{X}}\mu^\theta(x)(c^\theta(x,x')-c^{\theta+\epsilon}(x,x')) \\
&= -\sum_{x,x'\in\mathcal{X}}\mu^\theta(x)c^\theta(x,x^{\prime}) (G_1(\theta+\epsilon)-G_1(\theta))
+ \sum_{x,x'\in\mathcal{X}}\mu^\theta(x) (G_2(\theta+\epsilon)-G_2(\theta)) \\
&= -\sum_{x,x'\in\mathcal{X}}\mu^\theta(x)c^\theta(x,x') \Big( \epsilon^T \frac{\nabla_\theta c^\theta(x,x')}{c^\theta(x,x')}
+ \frac{1}{2}\epsilon^T \Big( \frac{\nabla_\theta^2 c^\theta(x,x')}{c^\theta(x,x')}
- \Big(\frac{\nabla_\theta c^\theta(x,x')}{c^\theta(x,x')}\Big)^2 \Big) \epsilon + R_2(\theta;x,x') \Big) \\
&+ \sum_{x,x'\in\mathcal{X}}\mu^\theta(x) \big( \epsilon^T \nabla_\theta c^\theta(x,x')
+  \frac{1}{2} \epsilon^T \nabla_\theta^2 c^\theta(x,x')\epsilon + \tilde{R}_2(\theta;x,x') \Big) \\
&= \frac{1}{2} \epsilon^T \sum_{x,x'\in\mathcal{X}}\mu^\theta(x) c^\theta(x,x') \Big(\frac{\nabla_\theta c^\theta(x,x')}{c^\theta(x,x')}\Big)^2 \epsilon
- \sum_{x,x'\in\mathcal{X}}\mu^\theta(x) \big( c^\theta(x,x') R_2(\theta;x,x') - \tilde{R}_2(\theta;x,x') \big)
\end{aligned}
\end{equation*}
where $R_2(\theta)$ and $\tilde{R}_2(\theta)$ are the remainder terms of $G_1$ and $G_2$, respectively.
\end{proof}

\subsection{Stochastic differential equations}
\label{SDE:app}
	We also compute the relative entropy rate for Ito diffusion processes. To avoid technical difficulties we impose
	following assumptions: we assume that the vector fields 
	$\DRIFTA(x)$, $\DRIFTB(x)\in\R^d$, $x\in\R^d$  and 
	the non-singular $\sigma(x)\in\R^{d\times d}$ are such that the Ito's stochastic differential equations
	\begin{eqnarray}
		&& d\PROCA = \DRIFTA(\PROCA) dt + \sigma(\PROCA) dW_t \COMMA  \label{entrate:SDE1} \\
		&& d\PROCB = \DRIFTB(\PROCB) dt + \sigma(\PROCB) dW_t         \label{entrate:SDE2}\COMMA
	\end{eqnarray}
	have a unique weak solution for initial conditions $\PROCAI\sim \nu_0(dx)$ and $\PROCBI \sim \mu_0(dx)$.
	Furthermore, we assume that the function 
	$$
	u(x) = \sigma^{-1}(x)(\DRIFTB(x) - \DRIFTA(x))
	$$
	is such that Novikov's condition $\EXPECT\left[\EXP{\frac{1}{2}\int_0^T |u(\PROCA)|^2\,dt}\right]<\infty$
	is satisfied \cite{Oksendal:00}.
	Under these assumptions we obtain explicit formula for the relative entropy rate of the stationary
	process $\{\PROCA\}_{t\geq 0}$ that is the solution of \VIZ{entrate:SDE1} with the initial condition $\PROCAI \sim \nu(dx)$, where
	$\nu(dx)$ is the invariant distribution.
	\begin{lemma}
		Let  $\{\PROCA\}_{t\geq 0}$ and $\{\PROCB\}_{t\geq 0}$ be the unique solutions of \VIZ{entrate:SDE1}-\VIZ{entrate:SDE2} with
		the initial conditions $\PROCAI\sim \nu_0(dx)$ and $\PROCBI \sim \mu_0(dx)$, where $\nu_0(dx)=\nu(dx)$ is the invariant distribution
		for the process $\{\PROCA\}_{t\geq 0}$. We define 
		$u(x) = \sigma^{-1}(x)(\DRIFTA(x) - \DRIFTB(x))$.
		Denoting $\PPTB$ and $\PPTA$ the corresponding path probability measures, the relative entropy is
		\begin{equation}\label{relent:SDE}
			\RELENT{\PPTB}{\PPTA} = \EXPECT_{\PPTB}\left[\frac{1}{2}\int_0^T  |u(\PROCA)|^2\,dt\right] + \RELENT{\nu_0}{\mu_0}\COMMA
		\end{equation}
		and the relative entropy rate $\ENTRATE{\PPTBB}{\PPTAA} \equiv \lim_{T\to\infty} \frac{1}{T} \RELENT{\PPTB}{\PPTA}$ is 
		\begin{equation}\label{RER:SDE}
			\ENTRATE{\PPTBB}{\PPTAA} = \EXPECT_\nu\left[\frac{1}{2}\|\DRIFTA - \DRIFTB\|^2_{\Sigma^{-1}}\right]\COMMA
		\end{equation}
		where $\|b\|_{\Sigma^{-1}} \equiv \sum_{i,j=1}^d \Sigma^{-1}_{ij}(x)b_i(x)b_j(x)$ is the norm on $\R^d$ defined by the diffusion matrix $\Sigma = \sigma(x)\sigma^T(x)$. 
	\end{lemma}
	
	\begin{proof}
		Under the assumptions on the stochastic differential equations it follows from Girsanov's Theorem, \cite{Oksendal:00}, that
		$\PPTB\ll\PPTA$
		$$
		\frac{d\PPTB}{d\PPTA}(\PROCA) = \frac{d\nu_0}{d\mu_0}(\PROCAI) \, \EXP{-\int_0^t u(\PROCAS)\,dW_s - \frac{1}{2}\int_0^t |u(\PROCAS)|^2\,ds} \PERIOD
		$$
		Furthermore, $B_t = \int_0^t u(\PROCAS)\, ds + W_t$ is Brownian motion under $\PPTB$. Thus we have
		$$
		\begin{aligned}
		&\RELENT{\PPTB}{\PPTA} = \EXPECT_{\PPTB}\left[ \log \frac{d\PPTB}{d\PPTA}\right] \\
		&= \RELENT{\nu_0}{\mu_0} +  \EXPECT_{\PPTB}\left[ -\int_0^T u(\PROCAS)\,dW_s - \frac{1}{2}\int_0^T |u(\PROCAS)|^2\,ds \right] \\
		&= \RELENT{\nu_0}{\mu_0} + \EXPECT_{\PPTB}\left[ -\int_0^T u(\PROCAS)\,(dB_s - u(\PROCAS)\,ds) - \frac{1}{2}\int_0^T |u(\PROCAS)|^2\,ds \right] \\
		&= \RELENT{\nu_0}{\mu_0} + \EXPECT_{\PPTB}\left[\frac{1}{2}\int_0^T |u(\PROCAS)|^2\,ds \right] \COMMA
		\end{aligned}
		$$
		where in the last identity we use $\EXPECT_{\PPTB}\left[ -\int_0^T u(\PROCAS)\,dB_s\right]=0$ as $B_t$ is Brownian motion under $\PPTB$.
		If $\PROCAI \sim \nu$ and thus the process $\{\PROCA\}_{t\geq 0}$ is stationary we have 
		$$
		\EXPECT_{\PPTB}\left[\frac{1}{2}\int_0^T |u(\PROCAS)|^2\,ds \right] = T \EXPECT_\nu \left[ \frac{1}{2}|u(x)|^2\right]\COMMA
		$$
		from which \VIZ{RER:SDE} follows.
	\end{proof}

\begin{lemma}
Let the drift term $a^\theta(x)$ be parametrized by $\theta\in\mathbb R$ and assume that the mapping
$\theta\rightarrow a^\theta(\cdot)$ is $C^2$. Let $P^\theta_{[0,T]}$ (resp. $\mu^\theta$) be the path
(resp. stationary) measure of the associated process. Then, the path FIM is
\begin{equation}
\FISHERR{P^\theta} = \mathbb{E}_{\mu^\theta}\left[ \nabla_{\theta}a^{\theta}(x)^T(\sigma\sigma^T)^{-1}(x)\nabla_{\theta}a^{\theta}(x) \right]\PERIOD
\label{path:FIM:SDE}
\end{equation}
\end{lemma}

\begin{proof}
Taylor's theorem for the drift term $a^\theta(\cdot)$ around $\theta$ reads
\begin{equation*}
a^{\theta+\epsilon}(x) = a^{\theta}(x) + \nabla_{\theta}a^{\theta}(x)\epsilon + R_1(\theta) \, ,
\end{equation*}
where $\nabla_{\theta}a^{\theta}(\cdot)$ is a $d\times k$ matrix containing all the first-order partial
derivatives of the drift vector (i.e., the Jacobian matrix) while the vector $R_1(\theta)$ is the remainder
term of the Taylor's theorem. Then, the relative entropy rate of the path probability measure
$P^\theta_{[0,T]}$ with respect to the perturbed path probability measure $P^{\theta+\epsilon}_{[0,T]}$
can be written as
\begin{equation*}
\begin{aligned}
&\RELENTR{P^\theta}{P^{\theta+\epsilon}}
= \frac{1}{2} \mathbb E_{\mu^\theta} \left[ \big|\sigma^{-1}(x)\big( a^{\theta+\epsilon}(x)-a^{\theta}(x)\big)\big|^2 \right] \\
&= \frac{1}{2} \mathbb E_{\mu^\theta} \left[\big(\nabla_{\theta}a^{\theta}(x)\epsilon + R_1(\theta;x)\big)^T
(\sigma\sigma^T)^{-1}(x) \big(\nabla_{\theta}a^{\theta}(x)\epsilon + R_1(\theta;x)\big) \right] \\
&= \frac{1}{2} \epsilon^T \mathbb E_{\mu^\theta}\left[ \nabla_{\theta}a^{\theta}(x)^T(\sigma\sigma^T)^{-1}(x)\nabla_{\theta}a^{\theta}(x) \right] \epsilon \\
&+ \epsilon^T \mathbb E_{\mu^\theta}\left[ \nabla_{\theta}a^{\theta}(x)^T(\sigma\sigma^T)^{-1}(x) R_1(\theta;x) \right] 
+ \frac{1}{2} E_{\mu^\theta}\left[ |\sigma^{-1}(x) R_1(\theta;x) |^2 \right]
\end{aligned}
\end{equation*}
from which  \VIZ{path:FIM:SDE} follows. %

\end{proof}

\bibliographystyle{siam}

\begin{thebibliography}{10}

\bibitem{Anderson:12}
{\sc D.~F. Anderson}, {\em {An efficient finite difference method for parameter
  sensitivities of continuous-time Markov chains}}, {SIAM J. Numerical
  Analysis}, {50} ({2012}), pp.~{2237--2258}.

\bibitem{AK:2013}
{\sc G.~Arampatzis and M.~Katsoulakis}, {\em Goal-oriented sensitivity analysis
  for lattice kinetic monte carlo simulations}, {J. Chem. Phys.}, 12 (2014),
  p.~124108.

\bibitem{AK:2014}
{\sc G.~Arampatzis and M.~A. Katsoulakis}, {\em Numerical estimation of
  integrated autocorrelation time}, in preparation,  (2014).

\bibitem{AKP:2014}
{\sc G.~{Arampatzis}, M.~A. {Katsoulakis}, and Y.~{Pantazis}}, {\em
  {Accelerated Sensitivity Analysis in High-Dimensional Stochastic Reaction
  Networks}}, PLOS ONE,  (2015).

\bibitem{Arampatzis:15}
\leavevmode\vrule height 2pt depth -1.6pt width 23pt, {\em {Pathwise
  Sensitivity Analysis in Transient Regimes}}, RMMC proceedings,  (2015).

\bibitem{Casella:2002}
{\sc G.~Casella and R.~Berger}, {\em Statistical Inference}, Duxbury advanced
  series in statistics and decision sciences, Thomson Learning, 2002.

\bibitem{Chowdhary:13}
{\sc K.~Chowdhary and P.~Dupuis}, {\em Distinguishing and integrating aleatoric
  and epistemic variation in uncertainty quantification}, ESAIM: Mathematical
  Modelling and Numerical Analysis, 47 (2013), pp.~635--662.

\bibitem{DZ}
{\sc A.~Dembo and O.~Zeitouni}, {\em Large deviations techniques and
  applications}, Applications of mathematics, Springer, New York, Berlin,
  Heidelberg, 1998.

\bibitem{Dupuis:97}
{\sc P.~Dupuis and R.~Ellis}, {\em A Weak Convergence Approach to the Theory of
  Large Deviations}, Wiley Series in Probability and Statistics, 1997.

\bibitem{Gardiner:85}
{\sc C.~Gardiner}, {\em Handbook of Stochastic Methods: for Physics, Chemistry
  and the Natural Sciences}, Springer, 1985.

\bibitem{Glynn:90}
{\sc P.~Glynn}, {\em {Likelihood ratio gradient estimation for stochastic
  systems}}, {Communications of the ACM}, {33} ({1990}), pp.~{75--84}.

\bibitem{Hairer:10}
{\sc M.~Hairer and A.~J. Majda}, {\em A simple framework to justify linear
  response theory}, Nonlinearity, 23 (2010), pp.~909--922.

\bibitem{Karatzas:91}
{\sc I.~Karatzas and S.~Shreve}, {\em Brownian Motion and Stochastic Calculus},
  Springer, 1991.

\bibitem{Kats:Plechac:13}
{\sc M.~A. Katsoulakis and P.~Plechac}, {\em {Information-theoretic tools for
  parametrized coarse-graining of non-equilibrium extended systems}}, {J. Chem.
  Phys.}, {139} ({2013}).

\bibitem{Kay:93}
{\sc S.~M. Kay}, {\em Funtamentals of {S}tatistical {S}ignal {P}rocessing:
  {E}stimation {T}heory}, Prentice-Hall, Englewood Cliffs, NJ, 1993.

\bibitem{Kim:07}
{\sc D.~Kim, B.~Debusschere, and H.~Najm}, {\em Spectral methods for parametric
  sensitivity in stochastic dynamical systems}, Biophysical Journal, 92 (2007),
  pp.~379--393.

\bibitem{Kipnis:99}
{\sc C.~Kipnis and C.~Landim}, {\em Scaling Limits of Interacting Particle
  Systems}, Springer-Verlag, 1999.

\bibitem{Kloeden:99}
{\sc P.~E. Kloeden and E.~Platen}, {\em Numerical Solution of Stochastic
  Differential Equations}, Springer-Verlag, 3rd Ed., 1999.

\bibitem{Komorowski:11}
{\sc M.~Komorowski, M.~J. Costa, D.~A. Rand, and M.~P.~H. Stumpf}, {\em
  {Sensitivity, robustness, and identifiability in stochastic chemical kinetics
  models}}, {Proc. Natl. Acad. Sci. USA}, {108} ({2011}), pp.~{8645--8650}.

\bibitem{Lehmann:98}
{\sc E.~Lehmann and G.~Casella}, {\em Theory of Point Estimation}, Springer,
  2nd~ed., 1998.

\bibitem{Li:12}
{\sc J.~Li and D.~Xiu}, {\em Computation of failure probability subject to
  epistemic uncertainty}, SIAM Journal on Scientific Computing, 34 (2012),
  pp.~A2946--A2964.

\bibitem{Limnios:01}
{\sc N.~Limnios and G.~Oprisan}, {\em Semi-Markov Processes and Reliability},
  Springer, 2001.

\bibitem{Liptser:77}
{\sc R.~S. Liptser and A.~N. Shiryaev}, {\em Statistics of Random Processes: {I
  \& II}}, Springer, 1977.

\bibitem{Liu:MC}
{\sc J.~S. Liu}, {\em {Monte {C}arlo strategies in scientific computing}},
  {Springer Series in Statistics}, Springer-Verlag, New York, 2001.

\bibitem{Maes:00}
{\sc C.~Maes, F.~Redig, and A.~V. Moffaert}, {\em On the definition of entropy
  production, via examples}, J. Math. Phys., 41 (2000), pp.~1528--1553.

\bibitem{Majda}
{\sc A.~J. Majda and B.~Gershgorin}, {\em {Quantifying uncertainty in climate
  change science through empirical information theory}}, {Proc. Natl. Acad.
  Sci. USA}, {107} ({2010}), pp.~{14958--14963}.

\bibitem{Mattingly:10}
{\sc J.~C. Mattingly, A.~Stuart, and M.~Tretyakov}, {\em Convergence of
  numerical time-averaging and stationary measures via {P}oisson equations},
  SIAM J. Numer. Anal., 48 (2010), pp.~552--577.

\bibitem{Nielsen:09}
{\sc F.~Nielsen and V.~Garcia}, {\em Statistical exponential families: {A}
  digest with flash cards}, arXiv.org:0911.4863,  (2009).

\bibitem{Oksendal:00}
{\sc B.~Oksendal}, {\em Stochastic Differential Equations: {A}n introduction
  with applications}, Springer-Verlag, 2000.

\bibitem{Pantazis:Kats:13}
{\sc Y.~Pantazis and M.~Katsoulakis}, {\em A relative entropy rate method for
  path space sensitivity analysis of stationary complex stochastic dynamics},
  J. Chem. Phys., 138 (2013), p.~054115.

\bibitem{PKV:2013}
{\sc Y.~Pantazis, M.~Katsoulakis, and D.~Vlachos}, {\em Parametric sensitivity
  analysis for biochemical reaction networks based on pathwise information
  theory}, BMC Bioinformatics, 14 (2013), p.~311.

\bibitem{Plyasunov:07}
{\sc S.~Plyasunov and A.~P. Arkin}, {\em Efficient stochastic sensitivity
  analysis of discrete event systems}, J. Comp. Phys., 221 (2007),
  pp.~724--738.

\bibitem{Qian:06}
{\sc H.~Qian}, {\em Open-system nonequilibrium steady-state: statistical
  thermodynamics, fluctuations and chemical oscillations}, J. Phys. Chem., 110
  (2006), pp.~15063--15074.

\bibitem{Rathinam:10}
{\sc M.~Rathinam, P.~W. Sheppard, and M.~Khammash}, {\em Efficient computation
  of parameter sensitivities of discrete stochastic chemical reaction
  networks}, J. Chem. Phys., 132 (2010), pp.~034103--(1--13).

\bibitem{Shannon:48}
{\sc C.~E. Shannon}, {\em A mathematical theory of communication}, Bell System
  Technical Journal, 27 (1948), pp.~379--423.

\bibitem{Khammash:12}
{\sc P.~Sheppard, M.~Rathinam, and M.~Khammash}, {\em {A pathwise derivative
  approach to the computation of parameter sensitivities in discrete stochastic
  chemical systems}}, {J. Chem. Phys.}, {136} ({2012}), p.~{034115}.

\bibitem{Sokal:96}
{\sc A.~D. Sokal}, {\em Monte carlo methods in statistical mechanics:
  Foundations and new algorithms}, in Lectures at the Carg\`ese Summer School,
  1996.

\bibitem{Tsybakov:08}
{\sc A.~Tsybakov}, {\em Introduction to Nonparametric Estimation}, Springer,
  2008.

\end{thebibliography}

\end{document}